\documentclass[11pt]{article}

\title{PFH spectral invariants and $C^\infty$ closing lemmas}
\author{Oliver Edtmair and Michael Hutchings\footnote{Partially supported by NSF grant DMS-2005437.}}
\date{}

\usepackage{amssymb}
\usepackage{latexsym}
\usepackage{amsmath}
\usepackage{amsthm}
\usepackage{amscd}
\usepackage{color}

\newcommand{\mc}[1]{{\mathcal #1}}

\numberwithin{equation}{section}

\newtheorem{theorem}{Theorem}[section]
\newtheorem{proposition}[theorem]{Proposition}
\newtheorem{corollary}[theorem]{Corollary}
\newtheorem{lemma}[theorem]{Lemma}
\newtheorem{lemma-definition}[theorem]{Lemma-Definition}

\theoremstyle{definition}
\newtheorem{definition}[theorem]{Definition}
\newtheorem{choice}[theorem]{Choice}
\newtheorem{remark}[theorem]{Remark}

\newtheorem{example}[theorem]{Example}

\newtheorem{notation}[theorem]{Notation}

\newcommand{\floor}[1]{\left\lfloor #1 \right\rfloor}

\renewcommand{\frak}{\mathfrak}

\newcommand{\C}{{\mathbb C}}
\newcommand{\Q}{{\mathbb Q}}
\newcommand{\R}{{\mathbb R}}

\newcommand{\Z}{{\mathbb Z}}
\newcommand{\F}{{\mathbb F}}

\newcommand{\op}{\operatorname}

\newcommand{\M}{\mc{M}}

\newcommand{\Ker}{\op{Ker}}

\newcommand{\tensor}{\otimes}

\newcommand{\union}{\bigcup}

\newcommand{\bpm}{\begin{pmatrix}}
\newcommand{\epm}{\end{pmatrix}}

\renewcommand{\epsilon}{\varepsilon}

\begin{document}

\maketitle

\begin{abstract}
We develop the theory of spectral invariants in periodic Floer homology (PFH) of area-preserving surface diffeomorphisms. We use this theory to prove $C^\infty$ closing lemmas for certain Hamiltonian isotopy classes of area-preserving surface diffeomorphisms. In particular, we show that for a $C^\infty$-generic area-preserving diffeomorphism of the torus, the set of periodic points is dense. Our closing lemmas are quantitative, asserting roughly speaking that for a given Hamiltonian isotopy, within time $\delta$ a periodic orbit must appear of period $O(\delta^{-1})$. We also prove a ``Weyl law'' describing the asymptotic behavior of PFH spectral invariants.
\end{abstract}

\setcounter{tocdepth}{2}

\section{Introduction}
\label{sec:intro}

Throughout this paper, let $\Sigma$ be a closed connected surface of genus $g$, and let $\omega$ be a symplectic (area) form on $\Sigma$. We are interested in (orientation-preserving) area-preserving diffeomorphisms $\phi:(\Sigma,\omega)\to(\Sigma,\omega)$. We are also interested in Hamiltonian isotopy classes in the set of all area-preserving diffeomorphisms of $(\Sigma,\omega)$; we denote the Hamiltonian isotopy class of $\phi$ by $[\phi]$.

\subsection{Closing lemmas}
\label{sec:cl}

Our convention is that a {\em periodic orbit\/} of $\phi$ with period $k$ is a set of $k$ distinct points in $\Sigma$ that are cyclically permuted by $\phi$. A {\em periodic point\/} is a point in a periodic orbit.

\begin{definition}
Let $\Phi$ be a Hamiltonian isotopy class of area-preserving diffeomorphisms of $(\Sigma,\omega)$.
We say that $\Phi$ satisfies the {\em $C^\infty$ generic density property\/} if for a $C^\infty$-generic area-preserving diffeomorphism $\phi\in\Phi$, the set of periodic points of $\phi$ is dense in $\Sigma$.
\end{definition}

It was proved by Asaoka-Irie \cite{ai} that the Hamiltonian isotopy class of the identity satisfies the $C^\infty$ generic density property. It is natural to ask which other Hamiltonian isotopy classes satisfy this property.

\begin{remark}
The $C^\infty$ generic density property fails for some Hamiltonian isotopy classes of area-preserving diffeomorphisms of $T^2$. In fact, it follows from a result of Herman \cite[Annexe, Thm.\ 2.2]{herman} that if $\phi$ is a Diophantine rotation of $T^2$, then there is a neighborhood of $\phi$ in the $C^\infty$ topology in the space of maps Hamiltonian isotopic to $\phi$ such that any map $\phi'$ in this neighborhood is smoothly conjugate to $\phi$, and hence has no periodic orbits. (Thanks to V. Humili\`ere for this reference.)
\end{remark}

One approach to proving the $C^\infty$ generic density property is to create periodic orbits through a given region by local perturbations:

\begin{definition}
Let $\Phi$ be a Hamiltonian isotopy class of area-preserving diffeomorphisms of $(\Sigma,\omega)$. We say that $\Phi$ satisfies the {\em $C^\infty$ closing property\/} if for every map $\phi\in\Phi$ and for every nonempty open set $\mc{U}\subset \Sigma$, there exists a $C^\infty$-small Hamiltonian isotopy supported in $\mc{U}$ from $\phi$ to $\phi'$ such that $\phi'$ has a periodic orbit intersecting $\mc{U}$.
\end{definition}

Standard arguments, see e.g.\ \cite[\S3]{irie1}, show:

\begin{lemma}
If the Hamiltonian isotopy class $\Phi$ satisfies the $C^\infty$ closing property, then it satisfies the $C^\infty$ generic density property.
\end{lemma}

One can now ask which Hamitonian isotopy classes satisfy the $C^\infty$ closing property. One of our main results is the following:

\begin{theorem}
\label{thm:closing}
(proved in \S\ref{sec:proofs})
Let $\Phi$ be a Hamiltonian isotopy class of area-preserving diffeomorphisms of $(\Sigma,\omega)$. Suppose that $\Phi$ is rational (Definition~\ref{def:rational}) and satisfies the $U$-cycle property (Definition~\ref{def:UCP}). Then $\Phi$ satisfies the $C^\infty$ closing property.
\end{theorem}

To explain the rationality hypothesis, we need to introduce a key actor in the story, the {\em mapping torus\/} of $\phi$. This is a three-manifold defined by
\begin{equation}
\label{eqn:mappingtorus}
Y_\phi = [0,1]\times\Sigma/\sim,\quad\quad\quad (1,x)\sim (0,\phi(x)).
\end{equation}
The mapping torus is a fiber bundle over $S^1=\R/\Z$ with fiber $\Sigma$. If $t$ denotes the $[0,1]$ coordinate on $[0,1]\times\Sigma$, then the vector field $\partial_t$ on $[0,1]\times\Sigma$ descends to a vector field on $Y_\phi$, which we also denote by $\partial_t$. Periodic orbits of the map $\phi$ of period $d$ correspond to simple periodic orbits of the vector field $\partial_t$ whose projection to $S^1$ has degree $d$. Since the map $\phi$ preserves the symplectic form $\omega$ on $\Sigma$, this form induces a fiberwise symplectic form $\omega$ on $Y_\phi$. The latter extends to a closed $2$-form $\omega_\phi$ on $Y_\phi$, characterized by $\omega_\phi(\partial_t,\cdot)=0$. 

We need to consider how the cohomology class $[\omega_\phi]\in H^2(Y_\phi;\R)$ depends on $\phi$.
Let $\{\phi_s\}_{s\in[0,1]}$ be a smooth isotopy of area-preserving diffeomorphisms of $(\Sigma,\omega)$, and suppose for simplicity that $\phi_s$ is constant for $s$ close to $0$ or $1$. (See \S\ref{sec:invariance} for a more general formalism for Hamiltonian isotopies.) We then obtain a diffeomorphism of mapping tori
\[
f:Y_{\phi_0}\stackrel{\simeq}{\longrightarrow} Y_{\phi_1}.
\]
This is induced by the diffeomorphism of $[0,1]\times\Sigma$ sending
\begin{equation}
\label{eqn:deff}
(t,x) \longmapsto (t,\phi_t^{-1}(\phi_0(x))).
\end{equation}
If the isotopy $\{\phi_s\}$ is Hamiltonian, then $f^*[\omega_{\phi_1}] = [\omega_{\phi_0}]\in H^2(Y_{\phi_0};\R)$.

\begin{definition}
\label{def:rational}
The Hamiltonian isotopy class $\Phi$ is {\em rational\/} if for $\phi\in\Phi$, the cohomology class $[\omega_\phi]\in H^2(Y_\phi;\R)$ is a real multiple of a class in the image of $H^2(Y_\phi;\Z)$.
\end{definition}

\begin{example}
\label{ex:rationaltorus}
Suppose that $\Sigma=T^2=\R^2/\Z^2$ and $\omega$ is the restriction of the standard symplectic form on $\R^2$. Any orientation-preserving diffeomorphism of $T^2$ is isotopic to the diffeomorphism induced by a linear map $A\in\op{SL}_2\Z$; see e.g.\ the introduction to \cite{cb}. It follows that any area-preserving diffeomorphism is Hamiltonian isotopic to a map of the form $\phi(x)=Ax+b$ where $A\in\op{SL}_2\Z$ and $b\in\R^2/\Z^2$. A computation using the Mayer-Vietoris sequence shows that there is a short exact sequence\footnote{Our convention is that the homology of a topological space is taken with $\Z$ coefficients by default unless otherwise stated.}
\[
0 \longrightarrow H_2(\Sigma) \longrightarrow H_2(Y_\phi) \stackrel{h}{\longrightarrow} \op{Ker}(A-I: \Z^2\circlearrowleft) \longrightarrow 0,
\]
where the first arrow is induced by inclusion of the fiber $\{0\}\times\Sigma$, and the map $h$ is given by the homology class in $H_1(T^2)$ of the intersection with $\{0\}\times\Sigma$. If $Z\in H_2(Y_\phi)$, then we have $\int_Z\omega_\phi \equiv \omega(h(Z),b) \mod \Z$. Thus $[\phi]$ is rational if and only if $\omega(v,b)\in\Q$ whenever $v\in\op{Ker}(A-I: \Z^2\circlearrowleft)$.

In particular, if $A=I$, that is if $\phi$ is smoothly isotopic to the identity, then $[\phi]$ is rational if and only if $b\in\Q^2/\Z^2$, namely $\phi$ is a rational rotation. If $A-I$ has rank one, then $\phi$ is isotopic to a power of a Dehn twist, and $[\phi]$ is rational when $\omega(v,b)\in\Q$ where $v$ is a generator of $\op{Ker}(A-I: \Z^2\circlearrowleft)$. In all other cases, $\phi$ is finite order or Anosov and $b_2(Y_\phi)=1$, so $[\phi]$ is automatically rational.
\end{example}

The $U$-cycle property is too technical to explain in the introduction, but we will see in Example~\ref{ex:US3} and Lemma~\ref{lem:Ucycle} below that every rational Hamiltonian isotopy class on $S^2$ or $T^2$ satisfies this property. Thus Theorem~\ref{thm:closing} implies:

\begin{corollary}
\label{cor:torus}
Let $\Phi$ be a Hamiltonian isotopy class of area-preserving diffeomorphisms of $S^2$ or $T^2$. If $\Phi$ is rational, then $\Phi$ satisfies the $C^\infty$ closing property.
\end{corollary}

Note that the unique Hamiltonian isotopy class of area-preserving diffeomorphisms of $S^2$ is rational. Any non-rational area-preserving diffeomorphism of $T^2$ can be perturbed to a rational one by a $C^\infty$-small (non-Hamiltonian) isotopy. Then as in \cite[Cor.\ 1.2]{ai}, we obtain:

\begin{corollary}
\label{cor:density}
For a $C^\infty$-generic area-preserving diffeomorphism of $T^2$, the set of periodic points is dense in $T^2$.
\end{corollary}

\begin{remark}
\label{rem:update1}
After the initial version of this paper was completed, it was shown in \cite{ucyclic} that every rational Hamiltonian isotopy class on a surface of any genus satisfies the $U$-cycle property (see Remark~\ref{rem:upgrade}). Thus the $U$-cycle hypothesis in Theorem~\ref{thm:closing} is redundant, and Corollary~\ref{cor:torus} also holds for surfaces of higher genus.
\end{remark}

\begin{remark}
\label{rem:update0}
It was shown in \cite{rohil}, which appeared simultaneously with the initial version of this paper, that Corollary~\ref{cor:density} holds for surfaces of higher genus. The argument in \cite{rohil} also shows that every rational Hamiltonian isotopy class satisfies the $C^\infty$ generic density property.
\end{remark}

\subsection{Quantitative closing lemmas}
\label{sec:qcl}

Going beyond the $C^\infty$ closing property, our methods also yield {\em quantitative\/} closing lemmas, asserting roughly that during a given Hamiltonian isotopy, within time $\delta$ a periodic orbit must appear with period $O(\delta^{-1})$. We now give some examples of precise quantitative statements that we can prove.

\begin{definition}
\label{def:ual}
Let $\mc{U}\subset \Sigma$ be a nonempty open set, let $l\in(0,1)$, and let $a\in(0,\op{area}(\mc{U}))$. A {\em $(\mc{U},a,l)$-admissible Hamiltonian\/} is a smooth function $H:[0,1]\times\Sigma\to\R$ such that:
\begin{itemize}
\item $H(t,x)=0$ for $t$ close to $0$ or $1$.
\item $H(t,x)=0$ for $x\notin\mc{U}$.
\item $H\ge 0$.
\item There is an interval $I\subset(0,1)$ of length $l$ and a disk $D\subset\mc{U}$ of area $a$ such that $H\ge 1$ on $I\times D$.
\end{itemize}
\end{definition}

Given any Hamiltonian $H:[0,1]\times\Sigma\to\R$ satisfying the first bullet point above, let $\{\varphi_t\}_{t\in[0,1]}$ denote the associated Hamiltonian isotopy (see \S\ref{sec:invariance} for conventions), and given an area-preserving diffeomorphism $\phi$ of $\Sigma$, write $\phi_H=\phi\circ\varphi_1$.

\begin{theorem}
\label{thm:quant0}
(proved in \S\ref{sec:proofs})
Let $\phi$ be an area-preserving diffeomorphism of $(S^2,\omega)$, write $A=\int_{S^2}\omega$, let $\mc{U}\subset S^2$ be a nonempty open set, and let $H$ be a $(\mc{U},a,l)$-admissible Hamiltonian. If $0<\delta\le al^{-1}$, then for some $\tau\in[0,\delta]$, the map $\phi_{\tau H}$ has a periodic orbit intersecting $\mc{U}$ with period $d$ satisfying
\begin{equation}
\label{eqn:quant0}
d\le \floor{Al^{-1}\delta^{-1}}.
\end{equation}
\end{theorem}

\begin{remark}
When $\delta = al^{-1}$ and $Aa^{-1}\notin\Z$, the bound \eqref{eqn:quant0} is sharp! That is, under the hypotheses of Theorem~\ref{thm:quant0}, one cannot prove the existence of a period orbit intersecting $\mc{U}$ with period less than $\floor{Aa^{-1}}$. The reason is that if $a=\op{area}(\mc{U})-\epsilon$ for $\epsilon>0$ sufficiently small, then the open sets $\phi^i(\mc{U})$ for $0\le i < \floor{Aa^{-1}}$ have total area less than $A$ and thus could be disjoint.
\end{remark}

We also obtain a slightly weaker inequality for rational area-preserving diffeomoprhisms of the torus:

\begin{theorem}
\label{thm:quant1}
(proved in \S\ref{sec:proofs})
Let $\phi$ be an area-preserving diffeomorphism of $(T^2,\omega)$, and write $A=\int_{T^2}\omega$. Suppose that the Hamiltonian isotopy class $[\phi]$ is rational, let $\Omega\in H^2(Y_\phi;\Z)$ be an integral cohomology class such that $[\omega_\phi]$ is a positive multiple of the image of $\Omega$, and let $d_0=\langle \Omega,[T^2]\rangle$. Let $\mc{U}\subset\Sigma$ be a nonempty open set, and let $H$ be a $(\mc{U},a,l)$-admissible Hamiltonian. If $0<\delta\le al^{-1}$, then for some $\tau\in[0,\delta]$, the map $\phi_{\tau H}$ has a periodic orbit intersecting $\mc{U}$ with period $d$ satisfying
\[
d\le d_0\left(\floor{Ad_0^{-1}l^{-1}\delta^{-1}} + 1\right).
\]
\end{theorem}

Theorems~\ref{thm:quant0} and \ref{thm:quant1} are special cases of a more general statement, Theorem~\ref{thm:quantitative} below, which is also applicable to higher genus surfaces. For some further developments on quantitative closing lemmas, which appeared after the initial version of this paper, see \cite{chaideztanny,edtmair,altspec}.

\subsection{Background and motivation for the proofs}
\label{sec:motivation}

The main precedent for Theorem~\ref{thm:closing} is the proof by Irie \cite{irie1} of the following $C^\infty$ closing lemma for contact forms on closed three-manifolds. See also the survey \cite{humiliere}.

\begin{theorem}[Irie]
\label{thm:irie}
Let $Y$ be a closed three-manifold, let $\lambda$ be a contact form on $Y$, and let $\mc{U}\subset Y$ be a nonempty open set. Then there exists a contact form $\lambda'$ which is $C^\infty$-close to $\lambda$ and agrees with $\lambda$ outside of $\mc{U}$, such that the Reeb vector field associated to $\lambda'$ has a periodic orbit intersecting $\mc{U}$. 
\end{theorem}

The proof uses the {\em embedded contact homology\/} of $(Y,\lambda)$; see \cite{bn} for detailed definitions. If $\lambda$ is nondegenerate, meaning that the periodic orbits of the Reeb vector field are nondegenerate, then the embedded contact homology $ECH(Y,\lambda)$ is the homology of a chain complex generated (over $\Z/2$) by certain finite sets of Reeb orbits with multiplicities, and whose differential counts certain $J$-holomorphic curves in $\R\times Y$ for a suitable almost complex structure $J$. Taubes \cite{taubes} proved that $ECH(Y,\lambda)$ is canonically isomorphic to a version of Seiberg-Witten Floer cohomology of $Y$ as defined by Kronheimer-Mrowka \cite{km}, and in particular depends only on the contact structure $\xi=\op{Ker}(\lambda)$. For any contact form $\lambda$ on $Y$, possibly degenerate, and for any nonzero class $\sigma\in ECH(Y,\xi)$, there is a ``spectral invariant'' $c_\sigma(Y,\lambda)\in\R$, which is the total period of a finite set of Reeb orbits with multiplicities homologically selected by ECH. The spectral invariants, unlike ECH, are highly sensitive to the contact form $\lambda$.

For the proof of Theorem~\ref{thm:irie}, we can assume without loss of generality that $Y$ is connected.
There is then a well-defined map $U:ECH(Y,\xi)\to ECH(Y,\xi)$, which is induced by a chain map counting certain $J$-holomorphic curves that are constrained to pass through a base point in $\R\times Y$. Define a {\em $U$-sequence\/} to be a sequence of nonzero classes $\{\sigma_k\}_{k\ge 1}$ in $ECH(Y,\xi)$ such that $U\sigma_{k+1}=\sigma_k$ for all $k\ge 1$. As explained for example in \cite[Lem.\ A.1]{infinity}, results of Kronheimer-Mrowka \cite{km} on Seiberg-Witten Floer homology imply that $U$-sequences always exist.

The key ingredient now is the following ``Weyl law'' for ECH spectral invariants proved\footnote{This was earlier proved in a special case in \cite{qech}, and later given a different proof by Sun \cite{sun}.} in \cite{vc}. 

\begin{theorem}
\label{thm:vc}
\cite{vc}
Let $Y$ be a closed connected three-manifold, let $\lambda$ be a contact form on $Y$, and let $\{\sigma_k\}_{k\ge 1}$ be a $U$-sequence in $ECH(Y,\xi)$. Then
\begin{equation}
\label{eqn:vc}
\lim_{k\to\infty}\frac{c_{\sigma_k}(Y,\lambda)^2}{k} = 2\op{vol}(Y,\lambda).
\end{equation}
\end{theorem}

\noindent
Here the {\em contact volume\/} is defined by
\[
\op{vol}(Y,\lambda) = \int_Y\lambda\wedge d\lambda.
\]

To prove Theorem~\ref{thm:irie}, one can define a smooth one-parameter family of contact forms $\{\lambda_t\}$ such that $\lambda_0=\lambda$, outside of $\mc{U}$ we have $\lambda_t=\lambda$, and $\frac{d}{dt}\op{vol}(Y,\lambda_t)> 0$. There must then exist Reeb orbits of $\lambda_t$ passing through $\mc{U}$ for arbitrarily small $t$. Otherwise, there exists $\delta>0$ such that $\lambda_t$ has no Reeb orbit passing through $\mc{U}$ for $t\in[0,\delta]$. One can deduce that each spectral invariant $c_\sigma(Y,\lambda_t)$ is independent of $t\in[0,\delta]$. It then follows from the Weyl law \eqref{eqn:vc} that $\op{vol}(Y,\lambda_t)$ is also independent of $t\in[0,\delta]$, which is a contradiction.

Returning to area-preserving surface diffeomorphisms: Asaoka-Irie proved that a $C^\infty$-generic Hamiltonian diffeomorphism of $(\Sigma,\omega)$ has dense periodic points by starting with a Hamiltonian diffeomorphism $\phi$ and constructing a contact three-manifold with an open book decomposition whose page is $\Sigma$ with a disk removed, and whose monodromy is a slight modification of $\phi$. One can then apply Theorem~\ref{thm:irie} to find a $C^\infty$-small perturbation of the contact form with dense Reeb orbits, and translate this back to a $C^\infty$-small perturbation of $\phi$ with dense periodic orbits.

It is not obvious how to extend the above argument to other Hamiltonian isotopy classes, because there are cohomological obstructions to defining the desired contact form. We will instead work more directly with {\em periodic Floer homology\/} (PFH). This is a theory which is defined analogously to ECH, but using periodic orbits of an area-preserving surface diffeomorphism instead of Reeb orbits of a contact form on a three-manifold. PFH is isomorphic to a version of Seiberg-Witten Floer cohomology of the mapping torus $Y_\phi$, as shown by Lee-Taubes \cite{lt}. Originally, PFH was defined before ECH; see \cite{pfh2,pfh3}. Since then there have been many applications of ECH to dynamics of Reeb vector fields in three dimensions and symplectic embeddings in four dimensions. Applications of PFH have only recently begun to appear, including the spectacular proof by Cristofaro-Gardiner, Humili\`ere, and Seyfaddini \cite{simple} that the group of compactly supported area-preserving homeomorphisms of the disk is not simple, and additional applications of PFH to area-preserving homeomorphisms of the two-sphere \cite{twosphere}.

We will prove a PFH analogue of the ``Weyl law'' \eqref{eqn:vc} in Theorem~\ref{thm:weyl} below, replacing the notion of ``$U$-sequence'' by a notion of ``$U$-cycle''. This Weyl law implies closing lemmas as in Theorem~\ref{thm:closing} (under slightly stronger hypotheses on the $U$-cycles), following Irie's proof of Theorem~\ref{thm:irie}.

\begin{remark}
Cristofaro-Gardiner, Prasad, and Zhang \cite{rohil} have independently proved a related Weyl law in PFH by different methods. This Weyl law (together with a nonvanishing result for Seiberg-Witten Floer cohomology) is used in \cite{rohil} to prove the generic density results described in Remark~\ref{rem:update0}.
\end{remark}

A Weyl law is really much stronger than necessary to detect the creation of periodic orbits. Indeed, a Weyl law implies that during a suitable perturbation, infinitely many spectral invariants change, with certain asymptotics; but to detect the creation of a periodic orbit, it suffices to show that a single spectral invariant changes by any nonzero amount. In \S\ref{sec:gap} we introduce a refinement of Irie's argument which, instead of a Weyl law, uses bounds on ``spectral gaps'' coming from ball packings in symplectic cobordisms, to detect the creation of periodic orbits. This method in fact leads to stronger, quantitative closing lemmas as in \S\ref{sec:qcl} above.

The rest of the paper is organized as follows. In \S\ref{sec:pfh} we review the definition of periodic Floer homology. In \S\ref{sec:invariance} we discuss invariance of PFH under Hamitonian isotopy; here we find it useful to relate a Hamiltonian isotopy to the graph of a function $Y_\phi\to\R$. In \S\ref{sec:spectral} we explain how to define the PFH spectral numbers we will use. In \S\ref{sec:ballpacking} we prove a key lemma which gives relations between PFH spectral invariants of different maps in the same Hamiltonian isotopy class arising from ball packings in symplectic cobordisms between the graphs of different Hamiltonians.  In \S\ref{sec:gap} we use this lemma to show how ``spectral gaps'' in PFH allow one to detect the creation of periodic orbits, with quantitative bounds. In \S\ref{sec:proofs} we use this machinery to prove all of our theorems stated above. Finally, in \S\ref{sec:weyl} we state and prove a Weyl law for PFH spectral invariants.

\paragraph{Acknowledgments.}

We thank the anonymous referees for detailed comments which helped us improve the clarity of the exposition.


\section{Periodic Floer homology}
\label{sec:pfh}

We now set up the version of PFH that we will be using, which one might call ``twisted PFH'', in more detail. Most of this material is also explained in \cite{pfh2,pfh3,ir,lt,simple}, although we will use a particular bookkeeping formalism of Novikov rings and reference cycles to keep track of areas of holomorphic curves. Apart from this bookkeeping, PFH is extremely similar to ECH, and we will refer to the lecture notes \cite{bn} on ECH for some definitions and basic results that do not differ significantly from the PFH case.

\subsection{PFH generators and holomorphic curves}

Let $(\Sigma,\omega)$ be a closed connected surface of genus $g$ with a symplectic (area) form, and let $\phi:(\Sigma,\omega)\to(\Sigma,\omega)$ be an area-preserving diffeomorphism. 

\begin{definition}
\label{def:orbitset}
An {\em orbit set\/} is a finite set of pairs $\alpha=\{(\alpha_i,m_i)\}$ such that:
\begin{itemize}
\item
The $\alpha_i$ are distinct periodic orbits of $\phi$.
\item
The $m_i$ are positive integers.
\end{itemize}
\end{definition}

Let $Y_\phi$ denote the mapping torus of $\phi$ as in \eqref{eqn:mappingtorus}. For an orbit set as above, regarding the periodic orbits $\alpha_i$ as embedded loops in $Y_\phi$, we define the homology class
\[
[\alpha] = \sum_im_i[\alpha_i]\in H_1(Y_\phi).
\]

A periodic orbit of $\phi$ of period $k$ is {\em nondegenerate\/} if for $x\in\Sigma$ in the periodic orbit, the derivative $d\phi^k_x:T_x\Sigma\to T_x\Sigma$ does not have $1$ as an eigenvalue.  A nondegenerate orbit as above is {\em hyperbolic\/} if $d\phi^k_x$ has real eigenvalues. We say that $\phi$ is nondegenerate if all of its periodic orbits (including multiply covered periodic orbits where the points in $\Sigma$ are not distinct) are nondegenerate; this holds for $C^\infty$ generic $\phi$ in any Hamiltonian isotopy class $\Phi$.

\begin{definition}
\label{def:generator}
Assume that $\phi$ is nondegenerate. A {\em PFH generator\/} is an orbit set $\alpha=\{(\alpha_i,m_i)\}$ such that 
$m_i=1$ whenever $\alpha_i$ is hyperbolic.
\end{definition}

\begin{remark}
The requirement that $m_i=1$ when $\alpha_i$ is hyperbolic is motivated by the relation with Seiberg-Witten theory, and some such condition is needed to obtain a topological invariant; see \cite[\S2.7]{bn}. This requirement is also used in the proof that the differential $\partial_J$ defined below satisfies $\partial_J^2=0$; see \cite[\S5.4]{bn} for explanation in the analogous situation of ECH.
\end{remark}

\begin{notation}
If $\gamma$ and $\gamma'$ are $1$-cycles\footnote{In this paper, a ``$1$-cycle'' will always be a finite integer linear combination of closed oriented $1$-dimensional submanifolds.} in $Y_\phi$ with $[\gamma]=[\gamma']\in H_1(Y_\phi)$, let $H_2(Y_\phi,\gamma,\gamma')$ denote the set of relative homology classes of $2$-chains $Z$ in $Y_\phi$ with $\partial Z = \gamma-\gamma'$. This is an affine space over $H_2(Y_\phi)$.
\end{notation}

To define the differential on the chain complex below, we will need to choose a generic almost complex structure $J$ on $\R\times Y_\phi$ satisfying the following conditions. To state them, let $E\to Y_\phi$ denote the vertical tangent bundle of $Y_\phi\to S^1$; this subbundle of $TY_\phi$ plays an analogous role in PFH to the contact structure $\xi$ in ECH.

\begin{definition}
\label{def:admissible}
An almost complex structure $J$ on $\R\times Y_\phi$ is {\em admissible\/} if:
\begin{itemize}
\item $J(\partial_s)=\partial_t$, where $s$ denotes the $\R$ coordinate on $\R\times Y_\phi$.
\item $J$ is independent of $s$, i.e.\ invariant under translation of the $\R$ factor in $\R\times Y_\phi$.
\item
$J$ sends $E$ to itself, rotating positively with respect to the fiberwise symplectic form $\omega$. This last condition means that if $v\in E$ and $v\neq 0$, then $\omega(v,Jv)>0$.
\end{itemize}
\end{definition}

Fix a generic admissible $J$ as above. We consider $J$-holomorphic curves of the form
\[
u : (C,j) \longrightarrow (\R\times Y_\phi,J),
\]
where the domain $C$ is a compact connected Riemann surface with finitely many punctures. Here $j$ denotes the (almost) complex structure on $C$, and the holomorphic curve equation is $J\circ du = du\circ j$. We assume that in a neighborhood of each puncture, the map $u$ is asymptotic to $\R$ cross a periodic orbit of $\phi$ as $s\to+\infty$ or $s\to-\infty$. We declare two such curves to be equivalent if they differ by composition with a biholomorphism of the domains. The curve $u$ is {\em multiply covered\/} if it factors through a branched cover of domains with degree greater than $1$; otherwise $u$ is {\em somewhere injective\/}. In the latter case, $u$ is an embedding except possibly for finitely many singular points; see \cite[\S5.1]{bn} for explanation in the analogous case of ECH. In the somewhere injective case, we can identify the holomorphic curve $u$ with its image in $\R\times Y_\phi$, which by abuse of notation we also denote by $C$.

We define a {\em $J$-holomorphic current\/} in $\R\times Y_\phi$ to be a finite formal linear combination
\[
\mathcal{C} = \sum_id_iC_i
\]
where the $C_i$ are distinct somewhere injective $J$-holomorphic curves as above, and the $d_i$ are positive integers. If $\alpha$ and $\beta$ are orbit sets with $[\alpha]=[\beta]$, let $\M^J(\alpha,\beta)$ denote the moduli space of $J$-holomorphic currents $
\mathcal{C}$ in $\R\times Y_\phi$ which as currents are asymptotic to $\alpha$ as $s\to+\infty$ and to $\beta$ as $s\to-\infty$. For $Z\in H_2(Y_\phi,\alpha,\beta)$, let $\M^J(\alpha,\beta,Z)$ denote the set of such currents that represent the relative homology class $Z$.  See \cite[\S3]{bn} for more precise definitions in the analogous case of ECH. Note that $\R$ acts on $\M^J(\alpha,\beta,Z)$ by translation of the $\R$ coordinate on $\R\times Y_\phi$.

We note for later use that the admissibility conditions on the almost complex structure $J$ imply the following:

\begin{itemize}
\item
If $\eta$ is a simple periodic orbit of $\partial_t$ in $Y_\phi$, then the ``trivial cylinder'' $\R\times\eta$ is an embedded $J$-holomorphic curve in $\R\times Y_\phi$.
\item
The restriction of $\omega_{\phi}$ to any $J$-holomorphic curve is pointwise nonnegative. Consequently,
\begin{equation}
\label{eqn:positivearea}
\M^J(\alpha,\beta,Z)\neq\emptyset \Longrightarrow \int_Z\omega_\phi\ge 0.
\end{equation}
\end{itemize}

Given orbit sets $\alpha=\{(\alpha_i,m_i)\}$ and $\beta=\{(\beta_j,n_j)\}$ with $[\alpha]=[\beta]$, and given $Z\in H_2(Y_\phi,\alpha,\beta)$, the {\em ECH index\/}\footnote{Perhaps here it should be called the ``PFH index''.} is defined to be
\begin{equation}
\label{eqn:I}
I(\alpha,\beta,Z) = c_\tau(\alpha,\beta,Z) + Q_\tau(\alpha,\beta,Z) + \sum_i\sum_{k=1}^{m_i}\op{CZ}_\tau(\alpha_i^k) - \sum_j\sum_{k=1}^{n_j}\op{CZ}_\tau(\beta_j^k) \in \Z.
\end{equation}
Here $\tau$ is a homotopy class of trivialization of the bundle $E$ over the orbits $\alpha_i$ and $\beta_j$, while $c_\tau$ denotes the relative first Chern class, $Q_\tau$ denotes the relative self-intersection number, and $\op{CZ}_\tau(\gamma^k)$ denotes the Conley-Zehnder index of the $k^{th}$ iterate of the periodic orbit $\gamma$ with respect to $\tau$. The ECH index $I(\alpha,\beta,Z)$ does not depend on the choice of trivialization $\tau$, although the individual terms on the right hand side do. For the proof of this fact, and for detailed definitions of the above notions, see \cite[\S3.3]{pfh2}, or \cite[\S2]{ir} for the more general case of stable Hamiltonian structures which includes both PFH and ECH.

\begin{example}
\label{ex:fiberindex}
If $\alpha=\beta$, then there is a canonical bijection $H_2(Y_\phi,\alpha,\alpha)\simeq H_2(Y_\phi)$, as both sets have the same definition. If $[\Sigma]\in H_2(Y_\phi)$ denotes the homology class of a fiber of $Y_\phi\to S^1$, then we have
\begin{equation}
\label{eqn:fiberindex}
I(\alpha,\alpha,[\Sigma]) = 2 (d - g + 1).
\end{equation}
This is because the first term in \eqref{eqn:I} here is
\[
c_\tau(\alpha,\alpha,[\Sigma])
=
\langle c_1(E),[\Sigma]\rangle = \langle c_1(T\Sigma),[\Sigma]\rangle = 2-2g.
\]
The second term in \eqref{eqn:I} is
\[
Q_\tau(\alpha,\alpha,[\Sigma]) = 2d.
\]
This holds because by \cite[Eq.\ (3.11)]{ir}, the self-intersection number $Q_\tau(\alpha,\alpha,[\Sigma])$ is twice the algebraic intersection number of $\R\times\alpha$ with a fiber. The third term in \eqref{eqn:I} here is zero because the sums are empty.
\end{example}

If $\mc{C}\in\M^J(\alpha,\beta,Z)$, we write $I(\mc{C})=I(\alpha,\beta,Z)$. A key property of the ECH index is the following; see e.g.\ \cite[Prop.\ 3.7]{bn} for the ECH case which is completely analogous.

\begin{proposition}
\label{prop:fundamental}
Suppose that $J$ is generic. Let $\alpha$ and $\beta$ be orbit sets with $[\alpha]=[\beta]\in H_1(Y_\phi)$. Suppose\footnote{See Remark~\ref{rem:admissible} for the reason why we assume that $[\alpha]\cdot[\Sigma]>g$.} that $[\alpha]\cdot[\Sigma]>g$. Then:
\begin{description}
\item{(a)}
If $I(\mc{C})=1$, then we can write $\mc{C}=\mc{C}_0+C_1$, where $\mc{C}_0$ is $\R$-invariant, i.e.\ a (possibly zero) finite linear combination of trivial cylinders, and $C_1$ is an embedded holomorphic curve of Fredholm index $1$, cut out transversely.
\item{(b)}
If $I(\mc{C})=2$ and if $\alpha$ and $\beta$ are PFH generators, then we can write $\mc{C}=\mc{C}_0+C_2$, where $\mc{C}_0$ is $\R$-invariant, and $C_2$ is an embedded holomorphic curve of Fredholm index $2$, cut out transversely.
\end{description}
\end{proposition}

Note that the embedded holomorphic curves $C_1$ and $C_2$ above live in moduli spaces of embedded $J$-holomorphic curves which have dimension one and two, respectively.

\subsection{The chain complex}

To define the PFH of $\phi$ in general, we need to keep track of some information about relative homology classes of holomorphic curves in $\R\times Y_\phi$. There are different options for how to do this, resulting in different versions of PFH. We will use a version with Novikov ring coefficients, which depends on the following choice:

\begin{choice}
\label{choice:G}
Let $\Ker([\omega_\phi])$ denote the kernel of $\langle [\omega_\phi],\cdot\rangle: H_2(Y_\phi)\to\R$. Below, fix a subgroup $G\subset\Ker([\omega_\phi])$. On a first reading it may be simplest to just consider the case $G=\{0\}$, although later we will find it convenient to choose $G=\Ker([\omega_\phi])$.
\end{choice}

\begin{definition}
Let $q$ be a formal variable and\footnote{One can also use coefficients in $\Z$ instead of $\Z/2$; it is explained in \cite[\S9]{obg2} how to set up orientations for the differential in ECH, and this carries over to PFH. However the applications in this paper, and all other applications of ECH and PFH so far, only need $\Z/2$ coefficients.} write $\F=\Z/2$. Let $\Lambda_G$ denote the Novikov ring consisting of formal sums
\[
\sum_{A\in H_2(Y_\phi)/G}p_Aq^A
\]
where $p_A\in\F$, such that for each $R\in\R$, there are only finitely many $A$ such that $p_A\neq 0$ and $\langle [\omega_\phi],A\rangle > R$.
\end{definition}

\begin{definition}
A {\em reference cycle\/} for $\phi$ is a $1$-cycle $\gamma$ in $Y_\phi$.  We define the {\em degree\/} $d(\gamma)=[\gamma]\cdot[\Sigma]\in\Z$, where $[\Sigma]\in H_2(Y_\phi)$ denotes the homology class of a fiber of $Y_\phi\to S^1$. Note that if $\alpha=\{(\alpha_i,m_i)\}$ is an orbit set with $[\alpha]=[\gamma]\in H_1(Y_\phi)$, then the total period of the orbits $\alpha_i$, counted with their multiplicities $m_i$, must equal the degree $d(\gamma)$, and in particular $d(\gamma)\ge 0$ if such an orbit set exists.
\end{definition}

\begin{definition}
Fix a subgroup $G$ as above and a reference cycle $\gamma$ for $\phi$. A {\em $(G,\gamma)$-anchored orbit set\/} is a pair $(\alpha,Z)$, where $\alpha$ is an orbit set with $[\alpha]=[\gamma]\in H_1(Y_\phi)$, and $Z\in H_2(Y_\phi,\alpha,\gamma)/G$. 

We define the {\em symplectic action\/} of $(\alpha,Z)$ by
\[
\mc{A}(\alpha,Z) = \int_Z\omega_\phi.
\]
This is well defined by our assumption that $G\subset\Ker([\omega_\phi])$.

When $\phi$ is nondegenerate, a {\em $(G,\gamma)$-anchored PFH generator\/} is a $(G,\gamma)$-anchored orbit set $(\alpha,Z)$ for which $\alpha$ is a PFH generator. 
\end{definition}

We can now define the periodic Floer homology $HP(\phi,\gamma,G)$, which is a $\Lambda_G$-module.

\begin{definition}
\label{def:chains}
If $\phi$ is nondegenerate and $\gamma$ is a reference cycle, define $CP(\phi,\gamma,G)$ to be the set of (possibly infinite) formal sums
\begin{equation}
\label{eqn:formalsum}
\sum_{\alpha,Z}n_{\alpha,Z}(\alpha,Z)
\end{equation}
where:
\begin{itemize}
\item
The sum is over $(G,\gamma)$-anchored PFH generators $(\alpha,Z)$.
\item
Each coefficient $n_{\alpha,Z}\in\F$.
\item
For each $R\in\R$, there are only finitely many $(\alpha,Z)$ such that $n_{\alpha,Z}\neq 0$ and $\mc{A}(\alpha,Z) > R$.
\end{itemize}
Then $CP(\phi,\gamma,G)$ is a $\Lambda_G$-module, with the $\Lambda_G$-action given by
\[
\sum_{A\in H_2(Y_\phi)/G}p_Aq^A \cdot \sum_{\alpha,Z}n_{\alpha,Z}(\alpha,Z) = 
 \sum_{\alpha,W}\left(\sum_{A\in H_2(Y_\phi)/G}p_An_{\alpha,W-A}\right)(\alpha,W).
\]
The finiteness conditions imply that the right hand side\footnote{One can also write the right hand side in more informal notation as $\sum_{A,\alpha,Z}p_An_{\alpha,Z}(\alpha,Z+A)$.} is a well defined element of $CP(\phi,\gamma,G)$.
\end{definition}

\begin{remark}
\label{rem:admissible}
In general, to define the differential when $d(\gamma)\le g$, one needs to choose a slight perturbation of an admissible almost complex structure, relaxing the last condition in Definition~\ref{def:admissible}. This is because if $J(E)=E$, then each fiber of $\R\times Y_\phi\to\R\times S^1$ is a $J$-holomorphic curve, which is not cut out transversely when $g>0$, and this interferes with compactness arguments to define the PFH differential when $d(\gamma)\le g$. See e.g.\ \cite[\S9.5]{pfh2}. For simplicity, {\bf we assume below that $d(\gamma)>g$\/} so that we can stick with admissible almost complex structures. The theory below can be extended to the case $d(\gamma)\le g$ with some additional work, but this will not be necessary for the applications here.
\end{remark}

\begin{definition}
For generic admissible $J$, we define a differential
\[
\partial_J: CP(\phi,\gamma,G) \longrightarrow CP(\phi,\gamma,G)
\]
by
\begin{equation}
\label{eqn:defdj}
\partial_J\sum_{\alpha,Z}n_{\alpha,Z}(\alpha,Z) =
\sum_{\beta,W}\left(\sum_\alpha\sum_{\substack{V\in H_2(Y_\phi,\alpha,\beta)\\I(\alpha,\beta,V)=1}}n_{\alpha,W+V}\#\frac{\M^J(\alpha,\beta,V)}{\R}\right)(\beta,W).
\end{equation}
Here the first sum on the right hand side is over $(G,\gamma)$-anchored PFH generators $(\beta,W)$, the second sum on the right hand side is over PFH generators $\alpha$ homologous to $\gamma$, and $\#$ denotes the mod $2$ count.
\end{definition}

\begin{lemma}
\label{lem:ddefined}
$\partial_J$ is well defined.
\end{lemma}

\begin{proof}
Assume that $J$ is generic. By Proposition~\ref{prop:fundamental}(a) and the compactness result\footnote{The compactness result and other results in \cite{pfh2} made an additional hypothesis of ``$d$-admissibility'', asserting that $(\phi,J)$ has a nice form near the periodic orbits of period at most $d(\gamma)$. This hypothesis is no longer needed thanks to the asymptotic analysis of Siefring \cite{siefring}.} of \cite[Thm.\ 1.8(a)]{pfh2}, given homologous PFH generators $\alpha$ and $\beta$ and given $R>0$, the set
\[
\bigcup_{\substack{V\in H_2(Y_\phi,\alpha,\beta)\\I(\alpha,\beta,V)=1\\\int_V\omega_\phi < R}}\frac{\mathcal{M}^J(\alpha,\beta,V)}{\R}
\]
is finite. In particular, for a fixed $V\in H_2(Y_\phi,\alpha,\beta)$ with $I(\alpha,\beta,V)=1$, the set $\mathcal{M}^J(\alpha,\beta,V)/\R$ is finite, so that the mod 2 counts in \eqref{eqn:defdj} are well defined.

To prove that the sum on the right hand side of \eqref{eqn:defdj} is a well defined element of $CP(\phi,\gamma,G)$, fix $\sum_{\alpha,Z}n_{\alpha,Z}(\alpha,Z)\in CP(\phi,\gamma,G)$. We need to show that for each real number $R$, there are only finitely many $(G,\gamma)$-anchored PFH generators $(\beta,W)$ with $\int_W\omega_\phi>R$ such that the sum in parentheses on the right hand side of \eqref{eqn:defdj} has any nonzero terms; and we need to show that for each such $(\beta,W)$, there are only finitely many nonzero terms.

By \eqref{eqn:positivearea}, if $\int_W\omega_\phi>R$ and $\M^J(\alpha,\beta,V)$ is nonempty, then $\int_{V}\omega_\phi\ge 0$, so $\int_{W+V}\omega_\phi > R$. Write $Z=W+V$. By the definition of $CP(\phi,\gamma,G)$, there are only finitely many $(G,\gamma)$-anchored PFH generators $(\alpha,Z)$ with $\int_Z\omega_\phi>R$ and $n_{\alpha,Z}\neq 0$. For each such pair, and for each of the finitely many PFH generators $\beta$ with $[\beta]=[\gamma]$, it follows from the aforementioned compactness result of \cite[Thm.\ 1.8(a)]{pfh2} that the union over $W$ with $\int_W\omega_\phi>R$ and $I(\alpha,\beta,Z-W)=1$ of the moduli spaces $\M^J(\alpha,\beta,Z-W)/\R$ is finite.
\end{proof}

It follows from minor modifications of \cite[Thm.\ 7.20]{obg1} (which applies to ECH) that $\partial_J^2=0$.

\begin{definition}
We define the {\em periodic Floer homology} $HP(\phi,\gamma,G)$ to be the homology of the chain complex $(CP(\phi,\gamma,G),\partial_J)$.
\end{definition}

The $\Lambda_G$-module $HP(\phi,\gamma,G)$ does not depend on the choice of $J$. That is, if $J_1$ and $J_2$ are admissible and generic, then there is a canonical isomorphism
\begin{equation}
\label{eqn:PsiJ1J2}
\Psi_{J_1,J_2}: H_*((CP(\phi,\gamma,G),\partial_{J_2}) \stackrel{\simeq}{\longrightarrow} H_*((CP(\phi,\gamma,G),\partial_{J_1}).
\end{equation}
These canonical isomorphisms have the properties that $\Psi_{J_1,J_2}\circ \Psi_{J_2,J_3}=\Psi_{J_1,J_3}$ when $J_3$ is admissible and generic, and $\Psi_{J_1,J_1}$ is the identity.

We have the canonical isomorphisms \eqref{eqn:PsiJ1J2} because by a special case\footnote{
The version of PFH that appears in \cite[Thm.\ 6.2]{lt} uses different notation and depends on a ``$(c_\Gamma,[\omega_\phi])$-complete local system for periodic Floer homology'', where $\Gamma=[\gamma]\in H_1(Y_\phi)$. Our version of PFH arises from such a local system, which assigns to each PFH generator $\alpha$ with $[\alpha]=\Gamma$ the $\Lambda_G$-module consisting of elements of $CP(\phi,\gamma,G)$ only involving $\alpha$. If $\beta$ is another PFH generator with $[\beta]=\Gamma$, then the local system assigns to each class $V\in H_2(Y_\phi,\alpha,\beta)$ the morphism of $\Lambda_G$-modules sending $\sum_Zn_Z(\alpha,Z)\mapsto \sum_Wn_{W+V}(\beta,W)$. The case $G=0$, which is called ``maximally-twisted coefficients'' in \cite{lt}, is further discussed in \cite[Cor.\ 6.1]{lt}.
}
of a theorem of Lee-Taubes \cite[Thm.\ 6.2]{lt}, there is a canonical isomorphism
\begin{equation}
\label{eqn:leetaubes}
H_*((CP(\phi,\gamma,G),\partial_J)) \simeq HM^{-*}(Y_\phi,\frak{s}_\Gamma,-r[\omega_\phi];\Lambda_G)
\end{equation}
for $r>>0$.
Here the right hand side is a version of Seiberg-Witten Floer cohomology as defined by Kronheimer-Mrowka \cite[Def.\ 30.2.3]{km}, perturbed using the cohomology class $-r[\omega_\phi]$, while $\frak{s}_\Gamma$ is a spin-c structure determined by $\Gamma=[\gamma]$. This version of Seiberg-Witten Floer cohomology is the homology of a chain complex which is generated over $\Lambda_G$ by solutions to the three-dimensional Seiberg-Witten equations on $Y_\phi$, perturbed using the closed $2$-form $r\omega_\phi$, modulo gauge transformations $Y_\phi\to S^1$; see \cite[\S1.2]{lt}. The differential counts solutions to similarly perturbed four-dimensional Seiberg-Witten equations on $\R\times Y_\phi$, modulo gauge transformations $\R\times Y_\phi\to S^1$ whose homotopy class in $[\R\times Y_\phi,S^1]=H_2(Y_\phi)$ is contained in the subgroup $G$. The proof of the isomorphism \eqref{eqn:leetaubes} shows that for $r>>0$, if we choose the metric in the Seiberg-Witten equations to be determined by $\omega_\phi$ and $J$, then there is an isomorphism on the chain level. In fact, according to \cite[Thm.\ 6.1]{lt}, if $r>>0$ then there is a bijection between the Seiberg-Witten solutions counted by the Seiberg-Witten Floer differential, and the $J$-holomorphic currents counted by the PFH differential. Modding out these Seiberg-Witten solutions by a restricted set of gauge transformations corresponds to keeping track of relative homology classes of holomorphic currents on the PFH side.

\begin{remark}
\label{rem:crc}
If $\gamma'$ is another reference cycle with $[\gamma]=[\gamma']$, and if $Z\in H_2(Y_\phi,\gamma,\gamma')/G$, then it follows from the definition \eqref{eqn:defdj} that there is an isomorphism of chain complexes
\[
\psi_Z: (CP(\phi,\gamma,G),\partial_J) \stackrel{\simeq}{\longrightarrow} (CP(\phi,\gamma',G),\partial_J)
\]
sending
\begin{equation}
\label{eqn:caniso}
\sum_{\alpha,W}n_{\alpha,W}(\alpha,W) \longmapsto \sum_{\alpha,W}n_{\alpha,W+Z}(\alpha,W).
\end{equation}
This induces an isomorphism of $\Lambda_G$-modules
\begin{equation}
\label{eqn:psiz}
\Psi_Z: HP(\phi,\gamma,G) \stackrel{\simeq}{\longrightarrow} HP(\phi,\gamma',G)
\end{equation}
depending only on the relative homology class $Z$. Thus, up to isomorphism, $HP(\phi,\gamma,G)$ depends only on the diffeomorphism $\phi$ and the homology class $[\gamma]$, and not on the choice of reference cycle $\gamma$.
\end{remark}

We will see in Proposition~\ref{prop:invariance} below that the isomorphism class of $HP(\phi,\gamma,G)$ is also invariant under Hamiltonian isotopy of $\phi$. Thus for a Hamiltonian isotopy class $\Phi$, a homology class $\Gamma\in H_1(Y_\phi)$, and a subgroup $G$ of $\Ker([\omega_\phi])$, we have a well-defined isomorphism class of $\Lambda_G$-modules $HP(\Phi,\Gamma,G)$.

\subsection{Examples of PFH}

\begin{example}
\label{ex:identity}
Let $\phi$ be the identity map on $\Sigma$. Although $\phi$ is degenerate, one can define its PFH to be the PFH of a nondegenerate Hamiltonian perturbation; see Remark~\ref{rem:degeneratePFH} below.

The mapping torus is given by $Y_\phi = S^1\times \Sigma$. We have
\[
H_2(Y_\phi) = H_2(\Sigma) \oplus \left(H_1(S^1)\tensor H_1(\Sigma)\right).
\]
It follows that the Novikov ring $\Lambda_G$ consists of formal sums
\begin{equation}
\label{eqn:idlambda}
\sum_{k\le k_0}a_kq^{k[\Sigma]}
\end{equation}
where the coefficients $a_k$ are elements of a group ring,
\[
a_k \in \F[(H_1(S^1)\tensor H_1(\Sigma))/G].
\]
Write $[S^1]=[S^1]\times\{\op{pt}\}\in H_1(Y_\phi)$. If $d$ is a nonnegative integer and if $\Gamma=d[S^1]$, then we can choose the reference cycle $\gamma$ to consist of $d$ circles of the form $S^1\times\{x\}$, and there is an isomorphism
\begin{equation}
\label{eqn:identity}
HP\left(\op{id}_\Sigma, d[S^1]\times\{x\},G\right) \simeq \op{Sym}^d H_*(\Sigma;\F) \tensor_\F \Lambda_G.
\end{equation}
Here $\op{Sym}^d$ denotes the degree $d$ part of the graded symmetric product; given a homogeneous basis of $H_*(\Sigma;\F)$, this is a vector space over $\F$ with a basis consisting of symmetric degree $d$ monomials in basis elements of $H_*(\Sigma;\F)$, where basis elements in $H_1(\Sigma;\F)$ cannot be repeated\footnote{It is also true (a special property of surfaces) that there is a canonical isomorphism $\op{Sym}^d H_*(\Sigma;\F) = H_*(\op{Sym}^d\Sigma;\F)$, although we do not need this.}.

To prove the isomorphism \eqref{eqn:identity}, one fixes $d$ and replaces the identity map with the time $1$ flow $\phi$ of a $C^2$-small autonomous Hamiltonian $H:\Sigma\to\R$ which is a Morse function. It follows from Definition~\ref{def:generator} that PFH generators in the class $\Gamma=d[S^1]$ correspond to degree $d$ symmetric monomials in critical points of $H$, where index $1$ critical points cannot be repeated. One can choose a metric $g_\Sigma$ on $\Sigma$ making the pair $(H,g_\Sigma)$ Morse-Smale, along with a corresponding almost complex structure $J$ on $\R\times Y_\phi$ for which Morse flow lines give rise to $J$-holomorphic cylinders. The $S^1$ symmetry of the mapping torus can be used to show that no other $J$-holomorphic curves contribute to the PFH differential; this argument is worked out in \cite{farris,nw} for the very similar problem of computing the ECH of prequantization bundles. In particular, if we choose $H$ to be a perfect Morse function (this means that the Morse homology differential vanishes, or equivalently here that there are exactly $2g+2$ critical points), then the PFH differential vanishes, and the chain complex agrees with the right hand side of \eqref{eqn:identity}. The isomorphism \eqref{eqn:identity} depends only on a choice of anchors for the degree $d$ PFH generators.

For a homology class $\Gamma\in H_1(S^1\times\Sigma)$ which is not of the form $d[S^1]$ for a nonnegative integer $d$, the PFH is zero, because after a small perturbation of the identity as above, there are no PFH generators in the class $\Gamma$.
\end{example}

Some more examples of PFH (more precisely untwisted PFH in the monotone case, see \S\ref{sec:monotone} below) are computed in \cite{pfh3} and \cite[Cor.\ 1.5]{lt}. For classes $\Gamma$ with $d=\Gamma\cdot[\Sigma]=1$, the PFH is closely related\footnote{One might expect that $d=1$ PFH and fixed point Floer homology are the same, both with the differential counting holomorphic cylinders. However when $g(\Sigma)>0$, in principle the PFH differential may count some additional holomorphic curves, due to the fact that $d\le g$ here; see Remark~\ref{rem:admissible}.} to the Floer homology for symplectic fixed points, which has been computed by Cotton-Clay \cite{cottonclay}.

\subsection{The $U$ map}
\label{sec:Umap}

There is also a well-defined map
\begin{equation}
\label{eqn:Umap}
U:HP(\phi,\gamma,G)\longrightarrow HP(\phi,\gamma,G).
\end{equation}
This is induced by a chain map which is defined analogously to the differential \eqref{eqn:defdj}; but instead of counting $I=1$ holomorphic currents modulo $\R$ translation, it counts $I=2$ holomorphic currents that are constrained to pass through a base point in $\R\times Y_\phi$.

To be precise, fix $y\in Y_\phi$ which is not on any periodic orbit of the vector field $\partial_t$. Given homologous PFH generators $\alpha$ and $\beta$, and given $Z\in H_2(Y_\phi,\alpha,\beta)$, define
\[
\mathcal{M}^J_y(\alpha,\beta,Z) = \left\{\mathcal{C}\in\mathcal{M}^J(\alpha,\beta,Z) \;\big|\; (0,y)\in\mathcal{C}\right\}.
\]
For a generic admissible $J$ we define a map
\[
U_{J,y}: CP(\phi,\gamma,G) \longrightarrow CP(\phi,\gamma,G)
\]
by
\[
U_{J,y}\sum_{\alpha,Z}n_{\alpha,Z}(\alpha,Z) =
\sum_{\beta,W}\left(\sum_\alpha\sum_{\substack{V\in H_2(Y_\phi,\alpha,\beta)\\I(\alpha,\beta,V)=2}}n_{\alpha,W+V}\#\M^J_y(\alpha,\beta,V)\right)(\beta,W).
\]
This map is well defined by an argument similar to the proof of Lemma~\ref{lem:ddefined}, using Proposition~\ref{prop:fundamental}(b). Similarly to the proof that $\partial_J^2=0$, the map $U_{J,y}$ is a chain map:
\[
\partial_J U_{J,y} = U_{J,y} \partial_J.
\]
We define the $U$ map \eqref{eqn:Umap} to be the map on homology induced by the chain map $U_{J,y}$. Since $Y_\phi$ is connected, the $U$ map \eqref{eqn:Umap} does not depend on the choice of base point $y$; one can use a path between two choices of base point $y$ to define a chain homotopy between the corresponding chain maps $U_{J,y}$. See \cite[\S2.5]{wh} for details in the analogous case of ECH. 

The $U$ map \eqref{eqn:Umap} does not depend on the choice of $J$ either, because under the Lee-Taubes isomorphism \eqref{eqn:leetaubes}, it corresponds to a $U$ map on Seiberg-Witten Floer homology defined by Kronheimer-Mrowka in \cite[\S25.3]{km}. Taubes proved an analogous statement for ECH in \cite[Thm.\ 1.1]{taubes5}, and the same argument works for the PFH case\footnote{It should be noted that Kronheimer-Mrowka and Taubes use different but equivalent definitions of the $U$ map on Seiberg-Witten Floer homology. Taubes defines the $U$ map from a chain map which counts Seiberg-Witten solutions for which $\alpha$ vanishes at the base point $(0,y)\in\R\times Y$. Here $\alpha$ denotes the component of the spinor in the $+i$ eigenspace of Clifford multiplication by $\lambda$; see \cite[\S1.b]{taubes5}. Kronheimer-Mrowka define the $U$ map from a chain map counting solutions to Seiberg-Witten moduli spaces after taking the cap product of the moduli spaces with a Cech cocycle representing the cohomology class in the configuration space corresponding to the generator of $H^2(\C P^\infty;\Z)$; see \cite[\S9.7]{km}. One can choose the Cech cocycle so that the two chain maps agree.}.

We will see in Proposition~\ref{prop:invariance} below that the $U$ map is invariant (in a certain sense) under Hamiltonian isotopy of $\phi$.

\begin{example}
\label{ex:US2}
Suppose that $\Sigma=S^2$ and $\phi$ is Hamiltonian isotopic to the identity. Let $d$ be a positive integer, and set $\gamma=d[S^1]\times\{x\}$ as in Example~\ref{ex:identity}. Here we must take $G=\{0\}$. Under the identification \eqref{eqn:identity}, denote the generators of $\op{Sym}^d H_*(S^2;\F)$, in increasing homological degree, by $e_{d,0},e_{d,1},\ldots,e_{d,d}$. Then a calculation as in \cite[\S4.1]{bn} shows that
\[
\begin{split}
Ue_{d,i} &= e_{d,i-1}, \quad\quad\quad i=1,\ldots,d,\\
Ue_{d,0} &= q^{-[S^2]}e_{d,d}.
\end{split}
\]
\end{example}

The above example has an important property which we now formalize.

\begin{definition}
Let $\phi$ be an area-preserving diffeomorphism of $(\Sigma,\omega)$, let $\gamma$ be a reference cycle for $\phi$, and let $G$ be as in Choice~\ref{choice:G}. We say that a nonzero element $\sigma\in HP(\phi,\gamma,G)$ is {\em $U$-cyclic\/} if there is a positive integer $m$ such that
\begin{equation}
\label{eqn:Ucyclic}
U^{m(d(\gamma)-g+1)}\sigma = q^{-m[\Sigma]}\sigma.
\end{equation}
We say that $\sigma$ is $U$-cyclic {\em of order $m$} if $m$ is the smallest positive integer with this property. (Here $\phi$ can be degenerate; see Remark~\ref{rem:degeneratePFH}.)
\end{definition}

\begin{remark}
In general, if $U^k\sigma=q^{-m[\Sigma]}\sigma$ for some $k$, then we must have $k=m(d(\gamma)-g+1)$. The reason is that in the nondegenerate case, $U$ counts holomorphic currents with ECH index $I=2$; while if $\alpha$ is any PFH generator with $\alpha\cdot[\Sigma]=d$, then by Example~\ref{ex:fiberindex} we have $I(\alpha,\alpha,[\Sigma]) = 2(d-g+1)$.
\end{remark}

\begin{example}
\label{ex:US3}
Suppose that $\Sigma=S^2$ and $\gamma=d[S^1]\times\{x\}$ where $d$ is a positive integer. Then it follows from Example~\ref{ex:US2} that $e_{d,i}$, and indeed every nonzero element of $HP(\phi,\gamma,\{0\})$, is $U$-cyclic of order $1$.
\end{example}

\begin{definition}
\label{def:UCP}
We say that the Hamiltonian isotopy class $[\phi]$ has the {\em $U$-cycle property\/} if there exist $U$-cyclic elements with arbitrarily large degree. That is, we require that for all positive integers $d_0$, there exist a subgroup $G\subset\Ker([\omega_\phi])$, a reference cycle $\gamma$ for $\phi$ with $d(\gamma)\ge d_0$, and a $U$-cyclic element $\sigma \in HP(\phi,\gamma,G)$. (This condition is invariant under Hamiltonian isotopy of $\phi$.)
\end{definition}

\subsection{Filtered PFH}

Fix a nondegenerate symplectomorphism $\phi$, a reference cycle $\gamma$, a group $G$ as in Choice~\ref{choice:G}, and a real number $L$. Define $CP^L(\phi,\gamma,G)$ to be the set of formal sums \eqref{eqn:formalsum} in $CP(\phi,\gamma,G)$ such that $\mc{A}(\alpha,Z)<L$ whenever $n_{\alpha,Z}\neq 0$. For a generic admissible $J$, it follows from \eqref{eqn:positivearea} that $CP^L(\phi,\gamma,G)$ is a subcomplex of $(CP(\phi,\gamma,G),\partial_J)$.

\begin{definition}
We define the {\em filtered PFH}, denoted by $HP^L(\phi,\gamma,G)$, to be the homology of the subcomplex $(CP^L(\phi,\gamma,G),\partial_J)$.
\end{definition}

Inclusion of chain complexes induces a map
\begin{equation}
\label{eqn:imath1}
\imath^L: HP^L(\phi,\gamma,G)\longrightarrow HP(\phi,\gamma,G).
\end{equation}
Similarly we have inclusion-induced maps
\begin{equation}
\label{eqn:imath2}
\imath^{L_1,L_2}: HP^{L_1}(\phi,\gamma,G)\longrightarrow HP^{L_2}(\phi,\gamma,G)
\end{equation}
for $L_1\le L_2$. With respect to these maps, $HP(\phi,\gamma,G)$ is the direct limit of $HP^L(\phi,\gamma,G)$ as $L\to\infty$.

The filtered homology $HP^L(\phi,\gamma,G)$, as well as the inclusion-induced maps \eqref{eqn:imath1} and \eqref{eqn:imath2}, do not depend on the choice of $J$. An analogous statement for ECH is proved in \cite[Thm.\ 1.3]{cc2}, and a similar argument applies here.

\subsection{The monotone case}
\label{sec:monotone}

We now recall two alternate versions of PFH which are defined in the following special situation, which is possible when $[\phi]$ is rational. This is in fact the only case that we need to consider in order to prove our theorems stated in \S\ref{sec:intro}.

\begin{definition}
\label{def:monotone}
For $\Gamma\in H_1(Y_\phi)$, we say that the pair $(\phi,\Gamma)$ is {\em monotone\/} if\footnote{This condition is an analogue of the following. In the context of Hamiltonian Floer homology, one says that a symplectic manifold $(X,\omega)$ is ``monotone'' if $[\omega]$ is a real multiple  (sometimes assumed to be nonnegative) of $c_1(TX)$ on $\pi_2(X)$. For both PFH and Hamiltonian Floer homology, monotonicity allows one to bound the symplectic area of holomorphic curves in terms of their index, which is a step towards obtaining finite counts. See e.g. \cite{hofersalamon}, which introduced the use of Novikov rings to define Hamiltonian Floer homology in some non-monotone cases.} the cohomology class $[\omega]\in H^2(Y_\phi;\R)$ is a real multiple of the image of the class
\[
c_1(E)+2\op{PD}(\Gamma)\in H^2(Y_\phi;\Z).
\]
\end{definition}

In this case, one can define a simpler, ``untwisted'' version\footnote{This is the original version of PFH from \cite{pfh2,pfh3}.} of PFH, which we denote here by $\overline{HP}(\phi,\Gamma)$. When $\phi$ is nondegenerate, this is the homology of a chain complex $\overline{CP}(\phi,\Gamma)$ which is freely generated over $\F$ by the PFH generators in the homology class $\Gamma$. For a generic admissible almost complex structure $J$, the differential is defined by
\begin{equation}
\label{eqn:untwisteddifferential}
\partial_J\alpha = \sum_\beta\sum_{\substack{V\in H_2(Y_\phi,\alpha,\beta)\\I(\alpha,\beta,V)=1}}\#\frac{\M^J(\alpha,\beta,V)}{\R}\beta.
\end{equation}

In the monotone case one can also define a twisted version of PFH without using a Novikov ring\footnote{This is analogous to the twisted ECH introduced in \cite[\S11.2]{t3}.}, which we denote here by $\widetilde{HP}(\phi,\gamma,G)$, where $\gamma$ is a reference cycle and $G$ is as in Choice~\ref{choice:G}. This version of PFH is a module over the group ring $\F[H_2(Y_\phi)/G]$. Again assuming that $\phi$ is nondegenerate, it is the homology of a chain complex $\widetilde{CP}(\phi,\gamma,G)$ which is freely generated over $\F$ by $(G,\gamma)$-anchored PFH generators, and whose differential is defined by
\begin{equation}
\label{eqn:twisteddifferential}
\partial_J(\alpha,Z) = \sum_\beta\sum_{\substack{V\in H_2(Y_\phi,\alpha,\beta)\\I(\alpha,\beta,V)=1}}\#\frac{\M^J(\alpha,\beta,V)}{\R}(\beta,Z-V).
\end{equation}
The differentials \eqref{eqn:untwisteddifferential} and \eqref{eqn:twisteddifferential} are well defined because when computing the differential of a generator, the monotonicity hypothesis implies that there is an upper bound on the integral of $\omega_\phi$ over all holomorphic currents that one needs to count, so that one obtains a finite count; compare Lemma~\ref{lem:ddefined}.

Suppose now that the reference cycle $\gamma$ is positively transverse to the fibers of $Y_\phi\to S^1$. A framing $\tau$ of $\gamma$ then induces a $\Z$-grading on $\widetilde{HP}(\phi,\gamma,G)$. The grading of a generator $(\alpha,Z)$ is defined by
\begin{equation}
\label{eqn:grading}
|(\alpha,Z)| = I(\alpha,\gamma,Z)
\end{equation}
where the right hand side is defined as in \eqref{eqn:I}, but with no Conley-Zehnder terms for $\gamma$. The grading \eqref{eqn:grading} descends to a $\Z/N$ grading on $\overline{HP}(\phi,\Gamma)$, where $N$ denotes the divisibility of $c_1(E)+2\op{PD}(\Gamma)$ in $\op{Hom}(H_2(Y_\phi),\Z)$; note that $N$ is an even integer.

The definition of the $U$ map carries over to $\overline{HP}$ and $\widetilde{HP}$, and with respect to the above gradings, it has degree $-2$.

Even in the monotone case, we will need to use a twisted version of PFH with reference cycles in order to define spectral invariants. We will later need the following relation between twisted and untwisted versions:

\begin{lemma}
\label{lem:twisted}
Suppose that $\phi$ is nondegenerate and $(\phi,\Gamma)$ is monotone, let $\gamma$ be a reference cycle with $[\gamma]=\Gamma$, and choose $G=\Ker([\omega_\phi])$. Then there is a noncanonical isomorphism of $\Lambda_G$-modules
\begin{equation}
\label{eqn:twisted}
HP(\phi,\gamma,G) \simeq \overline{HP}(\phi,\Gamma)\tensor_F \Lambda_G.
\end{equation}
Under the above isomorphism, if $d=d(\Gamma)$, then
\begin{equation}
\label{eqn:utai}
U^{d-g+1} \longleftrightarrow U^{d-g+1}\tensor q^{-[\Sigma]}.
\end{equation}
\end{lemma}

\begin{proof}
Let $A\in H_2(Y_\phi)/G$ be the unique class such that $\langle[\omega_\phi],A\rangle$ is positive and minimal. Then the Novikov ring $\Lambda_G$ is canonically identified with $\F((q^{-A}))$, namely the ring of Laurent series in $q^{-A}$ with coefficients in $\F$.

As in Remark~\ref{rem:crc}, we can assume without loss of generality that $\gamma$ is positively transverse to $E$. Choose a framing $\tau$ of $\gamma$ as needed to define a $\Z$-grading on $\widetilde{HP}(\phi,\gamma,G)$ and a $\Z/N$ grading on $\overline{HP}(\phi,\Gamma)$. It follows from the definitions that there is a canonical isomorphism
\begin{equation}
\label{eqn:modN}
\widetilde{HP}_i(\phi,\gamma,G) = \overline{HP}_{i\op{mod} N}(\phi,\Gamma).
\end{equation}
On the left hand side, multiplication by $q^{-A}$ shifts the grading down by $N$. It follows that $HP(\phi,\gamma,G)$ is canonically identified with the set of sequences $(\sigma_i)_{i\in\Z}$ where $\sigma_i\in \widetilde{HP}_i(\phi,\gamma,G)$ and $\sigma_i=0$ if $i$ is sufficiently large. If we choose a right inverse of the projection $\Z\to\Z/N$, then together with \eqref{eqn:modN} this defines an identification of the above set of sequences with $\overline{HP}(\phi,\Gamma)\tensor_F \Lambda_G$. This gives an isomorphism \eqref{eqn:twisted}.

To prove \eqref{eqn:utai}, we observe that under the isomorphism \eqref{eqn:twisted} constructed above,
\[
U^{N/2} \longleftrightarrow U^{N/2}\tensor q^{-A}.
\]
The positive\footnote{Recall from Remark~\ref{rem:admissible} that we are making the standing assumption that $d>g$.} integer $2(d-g+1)$ must be divisible by $N$, since $c_1(E)+2\op{PD}(\Gamma)$ evaluates to $2(d-g+1)$ on $[\Sigma]$. It follows that
\begin{equation}
\label{eqn:bm}
U^{d-g+1} \longleftrightarrow U^{d-g+1}\tensor q^{-(2(d-g+1)/N)A}.
\end{equation}
By monotonicity, we have $\langle c_1(E)+2\op{PD}(\Gamma),A\rangle = N$, and it follows that $(2(d-g+1)/N)A=[\Sigma]$ in $H_2(Y_\phi)/G$. Putting this into \eqref{eqn:bm} proves \eqref{eqn:utai}.
\end{proof}


\section{Invariance of PFH under Hamiltonian isotopy}
\label{sec:invariance}

We now work out how PFH and the additional structure on it defined above behave under Hamiltonian isotopy of $\phi$.

It is useful for our purposes to define a Hamiltonian isotopy of $\phi$ via a smooth map
\[
H:Y_\phi\longrightarrow\R.
\]
Under the projection $[0,1]\times\Sigma\to Y_\phi$, the map $H$ pulls back to a map $\widetilde{H}:[0,1]\times\Sigma\to\R$ satisfying $\widetilde{H}(1,x) = \widetilde{H}(0,\phi(x))$. For $t\in[0,1]$, let $H_t=\widetilde{H}(t,\cdot):\Sigma\to\R$, and let $X_{H_t}$ denote the associated Hamiltonian vector field on $\Sigma$; we use the sign convention $\omega(X_{H_t},\cdot) = dH_t$. Let $\{\varphi_t\}_{t\in[0,1]}$ denote the Hamiltonian isotopy defined by $\varphi_0=\op{id}_\Sigma$ and $\partial_t\varphi_t=X_{H_t}\circ\varphi_t$. We define $\phi_H=\phi\circ\varphi_1$.

As in \eqref{eqn:deff}, we have a diffeomorphism
\[
f_H : Y_\phi \stackrel{\simeq}{\longrightarrow} Y_{\phi_H}
\]
defined by the diffeomorphism of $[0,1]\times\Sigma$ sending
\[
(t,x) \longmapsto (t,\varphi_t^{-1}(x)).
\]
If $\gamma$ is a reference cycle in $Y_\phi$, let $\gamma_H$ denote its pushforward $(f_H)_\#\gamma$ in $Y_{\phi_H}$.

\begin{proposition}
\label{prop:invariance}
Let $\phi$ be a (possibly degenerate) area-preserving diffeomorphism of $(\Sigma,\omega)$, 
let $\gamma\subset Y_\phi$ be a reference cycle, and fix $G$ as in Choice~\ref{choice:G}. For $H_1,H_2:Y_\phi\to\R$ with $H_1 < H_2$ and $\phi_{H_1}$, $\phi_{H_2}$ nondegenerate, there is a canonical isomorphism
\begin{equation}
\label{eqn:Psi}
\Psi_{H_1,H_2}: HP(\phi_{H_2},\gamma_{H_2},G) \longrightarrow HP(\phi_{H_1},\gamma_{H_1},G)
\end{equation}
with the following properties:
\begin{description}
\item{(a)} If $H_2<H_3$ and if $\phi_{H_3}$ is also nondegenerate, then
\[
\Psi_{H_1,H_3} = \Psi_{H_1,H_2} \circ \Psi_{H_2,H_3} : HP(\phi_{H_3},\gamma_{H_3},G) \longrightarrow HP(\phi_{H_1},\gamma_{H_1},G).
\]
\item{(b)}
$U\circ \Psi_{H_1,H_2} = \Psi_{H_1,H_2}\circ U$.
\item{(c)}
The isomorphism \eqref{eqn:Psi} is the direct limit as $L\to\infty$ of canonical maps
\begin{equation}
\label{eqn:PsiL}
\Psi^L_{H_1,H_2}: HP^L(\phi_{H_2},\gamma_{H_2},G) \longrightarrow HP^{L+\Delta}(\phi_{H_1},\gamma_{H_1},G)
\end{equation}
where
\begin{equation}
\label{eqn:Delta}
\Delta = \int_\gamma (H_2-H_1)dt.
\end{equation}
\item{(d)}
If $H_2-H_1$ is a constant $C>0$, so that $\phi_{H_1}=\phi_{H_2}$ and $\gamma_{H_1} = \gamma_{H_2}$, then $\Psi_{H_1,H_2}^L = \imath^{L,L+dC}$ where $d=[\gamma]\cdot[\Sigma]$, see \eqref{eqn:imath2}. In particular,  $\Psi_{H_1,H_2}$ is the identity map.
\end{description}
\end{proposition}

\begin{proof}
We proceed in 6 steps.

{\em Step 1.\/}
To prepare to define the map $\Psi_{H_1,H_2}$, we construct a ``strong symplectic cobordism of stable Hamiltonian structures'' between $(Y_{\phi_{H_1}},\omega_{\phi_{H_1}})$ and $(Y_{\phi_{H_2}},\omega_{\phi_{H_2}})$ as follows.

Consider the ``symplectization'' of the mapping torus defined by
\[
X = \R\times Y_\phi
\]
with the symplectic form
\[
\omega_X = ds\wedge dt + \omega_\phi.
\]
Here $s$ denotes the $\R$ coordinate on $\R\times Y_\phi$.

Given $H:Y_\phi\to\R$, define an inclusion
\[
\begin{split}
\imath_H: Y_\phi &\longrightarrow \R\times Y_\phi,\\
z &\longmapsto (H(z),z).
\end{split}
\]
We can then identify the mapping torus $Y_{\phi_H}$ with a hypersurface in $\R\times Y_\phi$ via the inclusion
\begin{equation}
\label{eqn:inclusion}
\imath_H\circ f_H^{-1} : Y_{\phi_H} \longrightarrow \R\times Y_\phi.
\end{equation}
Note that there is a symplectomorphism
\begin{equation}
\label{eqn:sc4}
\R\times Y_{\phi_H}\longrightarrow \R\times Y_\phi
\end{equation}
induced by the symplectomorphism of $(\R\times [0,1]\times \Sigma,ds\wedge dt + \omega)$ sending
\[
(s,t,x) \longmapsto (s+H(t,x), t, \varphi_t(x)).
\]
The inclusion \eqref{eqn:inclusion} is the restriction of the symplectomorphism \eqref{eqn:sc4} to $\{0\}\times Y_{\phi_H}$. It follows that
\begin{equation}
\label{eqn:sc}
(\imath_H\circ f_H^{-1})^*\omega_X = \omega_{\phi_H}.
\end{equation}
We note also that under the inclusion \eqref{eqn:inclusion}, the reference cycle $\gamma_H$ corresponds to the graph of $H$ on $\gamma$ in $\R\times Y_\phi$.

Now if $H_1 < H_2$, define
\begin{equation}
\label{eqn:M}
M = \{(s,z) \in \R\times Y_\phi \mid H_1(z) \le s \le H_2(z)\}
\end{equation}
with the symplectic form $\omega_M=(\omega_X)|_M$. It follows from the above calculations that the boundary components of $M$ have neighborhoods in $M$ symplectomorphic to $[0,\epsilon)\times Y_{\phi_{H_1}}$ and $(-\epsilon,0]\times Y_{\phi_{H_2}}$, where the latter manifolds are equipped with the restrictions of the symplectic forms on the symplectizations of $Y_{\phi_{H_1}}$ and $Y_{\phi_{H_2}}$. Using these neighborhood identifications, we can glue to form the ``symplectization completion'' of $M$, which is a symplectic four-manifold
\begin{equation}
\label{eqn:completedcobordism}
\overline{M} = \left((-\infty,0]\times Y_{\phi_{H_1}}\right) \union_{Y_{\phi_{H_1}}} M \union_{Y_{\phi_{H_2}}} \left([0,\infty)\times Y_{\phi_{H_2}}\right).
\end{equation}
We note that there is a canonical symplectomorphism
\begin{equation}
\label{eqn:bookkeeping}
\overline{M}\simeq\R\times Y_\phi
\end{equation}
which is the inclusion on $M$, and which on the rest of \eqref{eqn:completedcobordism} is defined using the restrictions of the symplectomorphisms \eqref{eqn:sc4} for $H_1$ and $H_2$.

{\em Step 2.\/}
Suppose now that $\phi_{H_1}$ and $\phi_{H_2}$ are nondegenerate. Observe that $S=(\R\times \gamma)\cap M$ defines a 2-chain in the cobordism $M$ with $\partial S = \gamma_{H_2} - \gamma_{H_1}$. The cobordism $M$, together with the $2$-chain $S$, induces the desired map $\Psi_{H_1,H_2}$ in \eqref{eqn:Psi}, as a special case of a general construction of cobordism maps\footnote{This is related to the construction of cobordism maps on ECH in \cite[Thm.\ 1.9]{cc2}.} on PFH by Chen \cite[Thm.\ 1]{chen}. Chen's cobordism map in this case is defined by composing the Lee-Taubes isomorphism \eqref{eqn:leetaubes} on both sides with a Seiberg-Witten cobordism map
\begin{equation}
\label{eqn:swcm}
HM^{-*}(Y_{\phi_{H_2}},\frak{s}_\Gamma,-r[\omega_{\phi_{H_2}}];\Lambda_G) \longrightarrow HM^{-*}(Y_{\phi_{H_1}},\frak{s}_\Gamma,-r[\omega_{\phi_{H_1}}];\Lambda_G).
\end{equation}
Here $r>>0$ and $\Gamma=[\gamma]\in H_1(Y_\phi)$. The map \eqref{eqn:swcm} is induced by a chain map which counts solutions to the Seiberg-Witten equations on $\overline{M}$ perturbed using $r$ times the symplectic form on $\overline{M}$, similarly to the way the differential on the right hand side of \eqref{eqn:leetaubes} counts solutions to perturbed Seiberg-Witten equations on $\R\times Y_\phi$. The PFH cobordism maps in \cite[Thm.\ 1]{chen} satisfy a composition property which implies assertion (a), and they commute with the $U$ maps, giving assertion (b); these properties follow from corresponding properties of Seiberg-Witten cobordism maps.

{\em Step 3.}
To prove assertion (c), we will use the fact from \cite[Thm.\ 1]{chen} that the map $\Psi_{H_1,H_2}$ satisfies a crucial ``holomorphic curve axiom''. We now state this property.

Let $J_1$ and $J_2$ be almost complex structures on $\R\times Y_{\phi_{H_1}}$ and $\R\times Y_{\phi_{H_2}}$ as needed to define differentials $\partial_{J_1}$ on $CP(\phi_{H_1},\gamma_{H_1},G)$ and $\partial_{J_2}$ on $CP(\phi_{H_2},\gamma_{H_2},G)$. We can extend $J_1$ and $J_2$ to an almost complex structure $J$ on $\overline{M}$ whose restriction to $M$ is compatible with the symplectic form $\omega_M$.

Let $\alpha$ be an orbit set for $\phi_{H_2}$ and let $\beta$ be an orbit set for $\phi_{H_1}$. Define a {\em broken $J$-holomorphic current\/} in $\overline{M}$ from $\alpha$ to $\beta$ to be a tuple $(\mathcal{C}_{k_+},\mathcal{C}_{k_+-1},\ldots,\mathcal{C}_{k_-})$ where $k_+\ge 0\ge k_-$, and there are orbit sets $\alpha=\alpha(k+),\alpha(k_+-1),\ldots,\alpha(0)$ for $\phi_{H_2}$ and orbit sets $\beta(0),\beta(-1),\ldots,\beta(k_-)=\beta$ for $\phi_{H_1}$, such that:
\begin{itemize}
\item
$\mathcal{C}_i\in\M^{J_2}(\alpha(i),\alpha(i-1))/\R$ for $i>0$.
\item
$\mathcal{C}_0\in\M^J(\alpha(0),\beta(0))$. That is, $\mathcal{C}_0$ is a $J$-holomorphic current in $\overline{M}$ which as a current is asymptotic to $\alpha(0)$ as $s\to\infty$ on $[0,\infty)\times Y_{\phi_{H_2}}$, and asymptotic to $\beta(0)$ as $s\to-\infty$ on $(-\infty,0]\times Y_{\phi_{H_1}}$.
\item
$\mathcal{C}_i\in\M^{J_1}(\beta(i+1),\beta(i))/\R$ for $i<0$.
\end{itemize}

The holomorphic curves axiom now states that the map $\Psi$ is induced by a chain map
\[
\psi: (CP(\phi_{H_2},\gamma_{H_2},G),\partial_{J_2}) \longrightarrow (CP(\phi_{H_1},\gamma_{H_1},G),\partial_{J_1})
\]
with the following property. Similarly to \eqref{eqn:defdj}, we can write $\psi$ in the form
\begin{equation}
\label{eqn:chainmap}
\psi\sum_{\alpha,Z}n_{\alpha,Z}(\alpha,Z) = \sum_{\beta,W}\left(\sum_\alpha\sum_{\substack{V\in H_2(Y_\phi,\alpha,\beta)\\I(\alpha,\beta,V)=0}} n_{\alpha,W+V}m_{\alpha,\beta,V}\right)(\beta,W).
\end{equation}
Here the first sum on the right hand side is over $(G,\gamma_{H_1})$-anchored PFH generators $(\beta,W)$ for $\phi_{H_1}$, the second sum on the right hand side is over PFH generators $\alpha$ for $\phi_{H_2}$ in the homology class $[\gamma_{H_2}]$, and $m_{\alpha,\beta,V}\in\F$. The key property is now:
\begin{description}
\item{(*)} If the coefficient $m_{\alpha,\beta,V}\neq 0$, then there is a broken $J$-holomorphic current\footnote{The ``holomorphic curves axiom'' as stated in \cite[Thm.\ 1]{chen} implies a slightly weaker statement than (*), namely that for fixed $\alpha$ and $\beta$, if any of the coefficients $m_{\alpha,\beta,Z}$ is nonzero, then there exists a broken $J$-holomorphic current from $\alpha$ to $\beta$. The property (*) follows from the same argument, keeping track of the relative homology classes of the holomorphic currents.} in $\overline{M}$ from $\alpha$ to $\beta$ which, under the identification \eqref{eqn:bookkeeping}, represents the relative homology class $V$.
\end{description}
Note that the chain map $\psi$ is not canonical; see Remark~\ref{rem:psinotcanonical} below.

{\em Step 4.} 
We claim now that if $(\beta,W)$ is a $(G,\gamma_{H_1})$-anchored PFH generator for $\phi_{H_1}$, if $(\alpha,W+V)$ is a $(G,\gamma_{H_2})$-anchored PFH generator for $\phi_{H_2}$, and if there exists a broken $J$-holomorphic curent in $\overline{M}$ from $\alpha$ to $\beta$ in the relative homology class $V$, then
\begin{equation}
\label{eqn:step4}
\int_W\omega_{\phi_{H_1}} \le \int_{W+V}\omega_{\phi_{H_2}} + \Delta.
\end{equation}

To prove this, by \eqref{eqn:positivearea} we can assume without loss of generality that the $J$-holomorphic current has $k_+=k_-=0$, and thus consists of a single $J$-holomorphic current $\mathcal{C}\in\M^J(\alpha,\beta,V)$. We can also assume without loss of generality (by a slight modification of the cobordism) that $\mathcal{C}$ is transverse to $\partial M$. Under the decomposition \eqref{eqn:completedcobordism}, we can divide $\mathcal{C}$ into three pieces: let
\[
\begin{split}
\mathcal{C}_1 &= \mathcal{C} \cap \left((-\infty,0]\times Y_{\phi_{H_1}}\right),\\
\mathcal{C}_0 &= \mathcal{C} \cap M,\\
\mathcal{C}_2 &= \mathcal{C} \cap \left([0,\infty)\times Y_{\phi_{H_2}}\right).
\end{split}
\]
Since the almost complex structures $J_1$ and $J_2$ are admissible, as in \eqref{eqn:positivearea} we have
\begin{align}
\label{eqn:c1}
\int_{\mathcal{C}_1}\omega_{\phi_{H_1}} &\ge 0,\\
\label{eqn:c2}
\int_{\mathcal{C}_2}\omega_{\phi_{H_2}} &\ge 0.
\end{align}
Also, since $J$ is $\omega_M$-compatible, we have
\begin{equation}
\label{eqn:c0}
\int_{\mathcal{C}_0}\omega_M\ge 0.
\end{equation}

We now deduce \eqref{eqn:step4} by applying Stokes's theorem. To start, write $\eta_1=\mathcal{C}_1\cap (\{0\}\times Y_{\phi_{H_1}})$ and $\eta_2=\mathcal{C}_2\cap (\{0\} \times Y_{\phi_{H_2}})$. Then $\mathcal{C}_1$ projects, via the projection $(-\infty,0]\times Y_{\phi_{H_1}}\to Y_{\phi_{H_1}}$, to a relative homology class of $2$-chain $[\mathcal{C}_1]\in H_2(Y_{\phi_{H_1}},\eta_1,\beta)$. Likewise, $\mathcal{C}_2$ projects to a relative homology class of $2$-chain $[\mathcal{C}_2]\in H_2(Y_{\phi_{H_2}},\alpha,\eta_2)$. 

It follows from the homological assumption on $\mathcal{C}$ that in $M$, the $2$-cycle
\[
(\imath_{H_1}\circ f_{H_1}^{-1})_\#([\mathcal{C}_1]+W) + (\imath_{H_2}\circ f_{H_2}^{-1})_\#([\mathcal{C}_2]-W-V) + \mathcal{C}_0  - S
\]
is nullhomologous. Consequently the integral of the closed $2$-form $\omega_M$ over this $2$-cycle vanishes. By \eqref{eqn:sc}, this means that
\begin{equation}
\label{eqn:stokes}
\int_{\mathcal{C}_1+W}\omega_{\phi_{H_1}} + \int_{\mathcal{C}_2-W-V}\omega_{\phi_{H_2}} + \int_{\mathcal{C}_0}\omega_M - \Delta =0.
\end{equation}
Combining \eqref{eqn:c1}, \eqref{eqn:c2}, \eqref{eqn:c0}, and \eqref{eqn:stokes} proves \eqref{eqn:step4}.

{\em Step 5.\/}
We now prove (c). It follows from Step 4 that the chain map \eqref{eqn:chainmap} restricts to a chain map
\[
\psi^L:  (CP^L(\phi_{H_2},\gamma_{H_2},G),\partial_{J_2}) \longrightarrow (CP^{L+\Delta}(\phi_{H_1},\gamma_{H_1},G),\partial_{J_1})
\]
We define the map $\Psi^L_{H_1,H_2}$ in \eqref{eqn:PsiL} to be the map on filtered PFH induced by $\psi^L$. Although the chain map $\psi$ is defined only up to chain homotopy, the chain homotopies satisfy a version of the ``holomorphic curves axiom'' which implies that $\Psi^L_{H_1,H_2}$ depends only on $J_1$ and $J_2$. See e.g.\ \cite[Prop.\ 6.2]{calabi} for an analogous argument in the case of ECH. The map $\Psi^L_{H_1,H_2}$ does not depend on $J_1$ and $J_2$ either. An analogous statement for ECH was proved in \cite[Thm.\ 1.9]{cc2}, and this carries over to the case of PFH using the analysis of Chen \cite[Thm.\ 1]{chen}.

{\em Step 6.} Finally, the proof of property (d) follows the proof of the ``scaling'' property for ECH cobordism maps in \cite[Thm.\ 1.9]{cc2}.
\end{proof}

\begin{remark}
\label{rem:psinotcanonical}
Naively one would like to define the chain map \eqref{eqn:chainmap} by taking $m_{\alpha,\beta,V}$ to be a count of $I=0$ holomorphic currents in $\M^J(\alpha,\beta,V)$. Unfortunately it is not currently known in general\footnote{Such a count is possible in some special cases; see e.g.\ \cite[Thm.\ 2]{chen}, \cite{gerig}, and \cite{rooney}.} how to directly count $J$-holomorphic currents with $I=0$ in a completed cobordism, due to transversality difficulties with multiply covered holomorphic curves; see \cite[\S5.5]{bn} for explanation in the case of ECH, where there are similar issues. The actual chain map \eqref{eqn:chainmap} is defined instead by counting solutions to the Seiberg-Witten equations on $\overline{M}$, using the metric determined by $J$ and the symplectic form, and perturbed using a large multiple of the symplectic form as in \cite{lt}. The chain map \eqref{eqn:chainmap} is not canonical, because in cases where transversality fails for holomorphic curves, the chain map depends on additional small perturbations to the Seiberg-Witten equations needed to obtain transversality of the relevant moduli spaces of Seiberg-Witten solutions.
\end{remark}

\begin{remark}
\label{rem:canonicaliso}
In Proposition~\ref{prop:invariance}, if we drop the hypothesis that $H_1<H_2$, then there is still a canonical isomorphism \eqref{eqn:Psi}. One can define this isomorphism as $\Psi_{H_1,H_3}\circ \Psi_{H_2,H_3}^{-1}$ where $\phi_{H_3}$ is nondegenerate and $H_1,H_2<H_3$. By Proposition~\ref{prop:invariance}(a) and (d), this isomorphism does not depend on the choice of $H_3$.
\end{remark}

\begin{remark}
\label{rem:degeneratePFH}
If $\phi$ is degenerate, then we can define $HP(\phi,\gamma,G)$ by first perturbing $\phi$ to be nondegenerate via a Hamiltonian isotopy. By Remark~\ref{rem:canonicaliso}, the PFH modules for such perturbations of $\phi$ are all canonically isomorphic to each other.
\end{remark}


\section{Spectral invariants in PFH}
\label{sec:spectral}

We now define spectral invariants in PFH, analogously to the spectral invariants in ECH defined in \cite[Def.\ 4.1]{qech}.

\begin{definition}
\label{def:spectralinvariant}
Suppose that $\phi$ is nondegenerate, let $\gamma$ be a reference cycle, fix $G$ as in Choice~\ref{choice:G}, and let $\sigma$ be a nonzero class in $HP(\phi,\gamma,G)$. Define the {\em spectral invariant\/}
\[
c_\sigma(\phi,\gamma)\in\R
\]
to be the infimum over $L\in\R$ such that $\sigma$ is in the image of the inclusion-induced map \eqref{eqn:imath1}.
\end{definition}

We now establish some properties of the spectral invariants $c_\sigma$. We first consider the dependence of $c_\sigma$ on basic choices. Given a nonzero Novikov ring element $\lambda=\sum_{A\in H_2(Y_\phi)/G}p_Aq^A\in\Lambda_G$, define
\begin{equation}
\label{eqn:Novikovnorm}
|\lambda| = \max\{\langle [\omega_\phi],A\rangle \mid p_A\neq 0\}.
\end{equation}
Note that this maximum is well-defined by the definition of the Novikov ring $\Lambda_G$.

\begin{proposition}
\label{prop:spectralproperties}
Suppose $\sigma\in HP(\phi,\gamma,G)$ is nonzero.
\begin{description}
\item{(a)}
If $\lambda\in\Lambda_G$ is invertible\footnote{
A Novikov ring element $\lambda = \sum_{A\in H_2(Y_\phi)/G}p_Aq^A$ is invertible if and only if the maximum in \eqref{eqn:Novikovnorm} is realized by a unique class $A\in H_2(Y_\phi)/G$.
},
then
\[
c_{\lambda\sigma}(\phi,\gamma) = c_\sigma(\phi,\gamma) + |\lambda|.
\]
\item{(b)}
In the notation of Remark~\ref{rem:crc}, we have
\[
c_{\Psi_Z(\sigma)}(\phi,\gamma') = c_\sigma(\phi,\gamma) - \int_Z\omega.
\]
\end{description}
\end{proposition}

\begin{proof}
(a)
It follows from the definitions that multiplication by $\lambda$ induces an isomorphism of chain complexes
\[
(CP^L(\phi,\gamma,G),\partial_J) \stackrel{\simeq}{\longrightarrow} (CP^{L+|\lambda|}(\phi,\gamma,G),\partial_J).
\]
The induced isomorphism on homology
\begin{equation}
\label{eqn:lambdaiso}
HP^L(\phi,\gamma,G)\stackrel{\simeq}{\longrightarrow} HP^{L+|\gamma|}(\phi,\gamma,G)
\end{equation}
respects the maps \eqref{eqn:imath1}, and it follows that
\[
c_{\lambda\sigma}(\phi,\gamma) \le c_\sigma(\phi,\gamma) + |\lambda|.
\]
The inverse of the isomorphism \eqref{eqn:lambdaiso} is induced by multiplication of chains by $\lambda^{-1}$, and this implies the reverse inequality since $|\lambda^{-1}|=-|\lambda|$.

(b) This follows by a similar argument.
\end{proof}

We now begin to explore how spectral invariants behave under Hamiltonian isotopy, using the notation of Proposition~\ref{prop:invariance}.

\begin{proposition}
\label{prop:spectralchange}
Let $\phi$ be a (possibly degenerate) area-preserving diffeomorphism of $(\Sigma,\omega)$, let $\gamma$ be a reference cycle for $\phi$, let $H_1,H_2:Y_\phi\to\R$ with $H_1<H_2$, and suppose that $\phi_{H_1}$ and $\phi_{H_2}$ are nondegenerate. Let $\sigma_2$ be a nonzero class in $HP(\phi_{H_2},\gamma_{H_2},G)$. Write $\sigma_1=\Phi_{H_1,H_2}(\sigma_2) \in HP(\phi_{H_1},\gamma_{H_1},G)$. Then
\begin{equation}
\label{eqn:spectralchange}
c_{\sigma_1}(\phi_{H_1},\gamma_{H_1}) \le c_{\sigma_2}(\phi_{H_2},\gamma_{H_2}) + \int_\gamma(H_2-H_1)dt.
\end{equation}
\end{proposition}

\begin{proof}
Write $\Delta=\int_\gamma(H_2-H_1)dt$. By Proposition~\ref{prop:invariance}(c), for each real number $L$ we have a commutative diagram
\[
\begin{CD}
HP^L(\phi_{H_2},\gamma_{H_2},G) @>{\imath^L}>> HP(\phi_{H_2},\gamma_{H_2},G)\\
@V{\Psi^L_{H_1,H_2}}VV @VV{\Psi_{H_1,H_2}}V\\
HP^{L+\Delta}(\phi_{H_1},\gamma_{H_1},G) @>{\imath^{L+\Delta}}>> HP(\phi_{H_1},\gamma_{H_1},G).
\end{CD}
\]
It follows that if $\sigma_2$ is in the image of the top arrow, then $\sigma_1$ is in the image of the bottom arrow.
\end{proof}

\begin{remark}
\label{rem:constantshift}
If $H_2-H_1$ is a constant $C>0$, then equality holds in \eqref{eqn:spectralchange}:
\[
c_{\sigma_1}(\phi_{H_1},\gamma_{H_1}) = c_{\sigma_2}(\phi_{H_2},\gamma_{H_2}).
\]
This follows from the definitions and Proposition~\ref{prop:invariance}(d).
\end{remark}

\begin{corollary}
\label{cor:continuity}
\begin{description}
\item{(a)}
The definition of $c_\sigma(\phi,\gamma)$ has a unique extension to the case where $\phi$ is degenerate\footnote{Here the PFH of a degenerate map $\phi$ is defined by Remark~\ref{rem:degeneratePFH}.} such that the following continuity property holds: Let $\sigma\in HP(\phi,\gamma,G)$, let $H_1,H_2:Y_\phi\to\R$, and for $i=1,2$ let $\sigma_i$ denote the corresponding class in $HP(\phi_{H_i},\gamma_{H_i},G)$.  Then
\begin{equation}
\label{eqn:continuity}
\left|c_{\sigma_1}(\phi_{H_1},\gamma_{H_1}) - c_{\sigma_2}(\phi_{H_2},\gamma_{H_2})\right| \le d(\gamma)\max_{Y_\phi}|H_2-H_1|.
\end{equation}
\item{(b)}
The extended spectral invariants satisfy the conclusions of Propositions~\ref{prop:spectralproperties} and \ref{prop:spectralchange}.
\end{description}
\end{corollary}

\begin{proof}
This follows from Proposition~\ref{prop:spectralchange} and Remark~\ref{rem:constantshift} using the formal procedure in \cite[\S3.1]{qech} and \cite[\S2.5]{onetwo}.
\end{proof}

The spectral invariants $c_\sigma$ have the following ``spectrality'' property.

\begin{proposition}
\label{prop:representation}
Let $\phi$ be an area-preserving diffeomorphism of $(\Sigma,\omega)$, possibly degenerate, and suppose that $[\phi]$ is rational. Then for any reference cycle $\gamma$, any $G$ as in Choice~\ref{choice:G}, and any nonzero class $\sigma\in HP(\phi,\gamma,G)$, there exists a $(G,\gamma)$-anchored orbit set $(\alpha,Z)$ such that
\[
c_\sigma(\phi,\gamma) = \mc{A}(\alpha,Z).
\]
\end{proposition}

\begin{proof}
To start, note that since $[\phi]$ is rational, the set of values that $[\omega]\in H^2(Y_\phi;\R)$ takes on $H_2(Y_\phi)$ is discrete.

Suppose first that $\phi$ is nondegenerate. Then there are only finitely many PFH generators in the homology class $[\gamma]$. Let $S$ denote the set of actions of $(G,\gamma)$-anchored PFH generators; then the set $S$ is discrete. If $L<L'$ and the interval $[L,L')$ does not intersect $S$, then the inclusion-induced map
\[
\imath^{L,L'} : HP^L(\phi,\gamma,G) \longrightarrow HP^{L'}(\phi,\gamma,G)
\]
is an isomorphism, since it is induced by an isomorphism of chain complexes. It then follows from Definition~\ref{def:spectralinvariant} that $c_\sigma(\phi,\gamma)\in S$.

Suppose now that $\phi$ is degenerate. Let $\{H_i\}_{i\ge 1}$ be a sequence of Hamiltonians converging to $0$ in $C^\infty$ such that each $\phi_{H_i}$ is nondegenerate. Let $\sigma_i$ denote the class in $HP(\phi_{H_i},\gamma_{H_i},G)$ corresponding to $\sigma$ under the canonical isomorphism given by Remarks~\ref{rem:canonicaliso} and \ref{rem:degeneratePFH}. By the continuity in \eqref{eqn:continuity}, we have
\[
c_\sigma(\phi,\gamma) = \lim_{i\to\infty}c_{\sigma_i}(\phi_{H_i},\gamma_{H_i}).
\]
By the nondegenerate case, for each $i$ there exists a $(G,\gamma_{H_i})$-anchored PFH generator $(\alpha(i),Z(i))$ for $\phi_{H_i}$ such that $c_{\sigma_i}(\phi_{H_i},\gamma_{H_i})=\mc{A}(\alpha(i),Z(i))$. Since $\Sigma$ is compact and each periodic orbit in each $\alpha(i)$ has period at most $d(\gamma)$, by the Arzela-Ascoli theorem we can pass to a subsequence so that $\alpha(i)$ converges in $C^\infty$ to an orbit set $\alpha$ for $\phi$. Then the distance from the sequence $\mc{A}(\alpha(i),Z(i))$ to the set $\{\mc{A}(\alpha,Z)\mid Z\in H_2(Y,\alpha,\gamma)/G\}$ limits to zero. Since the latter set is discrete by our rationality hypothesis, the sequence $(\mc{A}(\alpha(i),Z(i)))$ converges to an element of it. 
\end{proof}

\begin{remark}
Without the hypothesis that $[\phi]$ is rational, Proposition~\ref{prop:representation} still holds if $\phi$ is nondegenerate, by \cite[Thm.\ 1.4]{usher}; see also \cite{oh}. However we do not know whether Proposition~\ref{prop:representation} extends to the case where $[\phi]$ is not rational and $\phi$ is degenerate.
\end{remark}


\section{The ball packing lemma}
\label{sec:ballpacking}

We now prove a key fact, Lemma~\ref{lem:key} below, which will be needed for the proofs of the main theorems. This lemma gives a relation between the PFH spectral invariants of two different Hamiltonian perturbations of $\phi$.

To state the lemma, recall that a four-dimensional {\em Liouville domain\/} is a compact symplectic four-manifold $(X,\omega)$ with boundary $Y$ such that there exists a $1$-form $\lambda$ on $X$ for which $d\lambda=\omega$ and $\lambda|_Y$ is a contact form on $Y$. We further require that the orientation of $Y$ given by $\lambda\wedge d\lambda$ agrees with the boundary orientation of $X$. We allow $X$ to be disconnected.

If $(X,\omega)$ is a Liouville domain as above, its {\em alternative ECH capacities\/} are a sequence of real numbers
\[
0= c_0^{\op{Alt}}(X,\omega) < c_1^{\op{Alt}}(X,\omega) \le c_2^{\op{Alt}}(X,\omega) \le \cdots \le \infty
\]
defined in \cite{altech}.
To briefly review the definition, when $(X,\omega)$ is nondegenerate, meaning that all Reeb orbits of the contact form $\lambda|_Y$ are nondegenerate, we define
\begin{equation}
\label{eqn:ckalt}
c_k^{\op{Alt}}(X,\omega) = \sup_{\substack{J\in\mathcal{J}(\overline{X})\\ \mbox{\scriptsize $x_1,\ldots,x_k\in X$ distinct}}} \inf_{u\in\mathcal{M}^J(\overline{X};x_1,\ldots,x_k)} \mathcal{E}(u).
\end{equation}
Here $\overline{X} = X\cup_Y ([0,\infty)\times Y)$, similarly to \eqref{eqn:completedcobordism}. The notation $\mathcal{J}(\overline{X})$ indicates the set of almost complex structures $J$ on $\overline{X}$ such that $J$ is $\omega$-compatible on $X$, and on $[0,\infty)\times Y$, analogously to Definition~\ref{def:admissible}, $J$ is independent of the $[0,\infty)$ coordinate $s$, sends $\partial_s$ to the Reeb vector field, and sends the contact structure $\op{Ker}(\lambda)$ to itself, rotating positively with respect to $d\lambda$. Furthermore, $\mathcal{M}^J(\overline{X};x_1,\ldots,x_k)$ denotes the moduli space of $J$-holomorphic curves $u$ in $\overline{X}$ for which the domain is a compact (possibly disconnected) Riemann surface with finitely many punctures near which $u$ is asymptotic to Reeb orbits as $s\to\infty$. We require that $u$ is nonconstant on each component of the domain. Finally, $\mathcal{E}(u)$ denotes the sum over the punctures of $u$ of the symplectic action (period) of the corresponding Reeb orbit.

It is shown in \cite{altech} that $c_k^{\op{Alt}}$ has a unique $C^0$-continuous extension to degenerate Liouville domains. We will need the following examples. For $r>0$, define the ball
\[
B(r) = \{z\in\C^2\mid \pi|z|^2\le r\}
\]
with the restriction of the standard symplectic form on $\C^2=\R^4$. Note that the Euclidean volume of the ball is given by
\begin{equation}
\label{eqn:ballvolume}
\op{vol}(B(r)) = \frac{r^2}{2}.
\end{equation}
It is shown in \cite[Thm.\ 6]{altech} that the capacities of a ball are given by
\begin{equation}
\label{eqn:ckball}
c_k^{\op{Alt}}(B(r))=dr,
\end{equation}
where $d$ is the unique nonnegative integer such that
\[
d^2+d \le 2k \le d^2+3d.
\]
It is also shown in \cite[Thm.\ 6]{altech} that the capacities of a disjoint union are given by
\begin{equation}
\label{eqn:ckunion}
c_k^{\op{Alt}}\left(\coprod_{i=1}^n (X_i,\omega_i) \right) = \max_{k_1+\cdots+k_n=k}\sum_{i=1}^nc_{k_i}^{\op{Alt}}(X_i,\omega_i).
\end{equation}

For the above examples, the alternative ECH capacities $c_k^{\op{Alt}}$ agree with the original ECH capacities $c_k^{\op{ECH}}$ defined in \cite{qech}.
The calculation in \cite[Prop.\ 8.4]{qech} deduces from \eqref{eqn:ballvolume}, \eqref{eqn:ckball}, and \eqref{eqn:ckunion} that if $X$ is a finite disjoint union of balls, then\footnote{This calculation enters into the proof of the ECH Weyl law \eqref{eqn:vc} for the tight contact structure on $S^3$.}
\begin{equation}
\label{eqn:echvol}
\lim_{k\to\infty}\frac{c_k^{\op{Alt}}(X)^2}{k} = 4\op{vol}(X).
\end{equation}

Returning to PFH, to simplify notation we use the following convention:

\begin{notation}
\label{spectralnotation}
If $\sigma\in HP(\phi,\gamma,G)$, and if $H:Y_\phi\to\R$, let $\sigma_H\in HP(\phi_H,\gamma_H,G)$ denote the class corresponding to $\sigma$ under the canonical isomorphism given by Remarks~\ref{rem:canonicaliso} and \ref{rem:degeneratePFH}. Write
\[
c_\sigma(\phi,\gamma,H) = c_{\sigma_H}(\phi_H,\gamma_H)\in\R.
\]
\end{notation}

\begin{lemma}
\label{lem:key}
Let $\phi$ be an area-preserving diffeomorphism of $(\Sigma,\omega)$, and let $H_1,H_2:Y_\phi\to\R$ with $H_1\le H_2$. Let $\gamma$ be a reference cycle for $\phi$, let $\sigma\in HP(\phi,\gamma,G)$, let $k$ be a nonnegative integer, and suppose that $U^k\sigma\neq 0$. Let $(X,\omega)$ be a compact Liouville domain and suppose there exists a symplectic embedding of $(X,\omega)$ into the cobordism $M$ from \eqref{eqn:M}. Then
\[
c_{U^k\sigma}(\phi,\gamma,H_1) \le c_\sigma(\phi,\gamma,H_2) + \int_\gamma(H_2-H_1)dt - c_k^{\op{Alt}}(X,\omega).
\]
\end{lemma}

\begin{remark}
\label{rem:oneball}
When $k=0$, Lemma~\ref{lem:key} reduces to the inequality \eqref{eqn:spectralchange}.

For most of our applications, the only case of Lemma~\ref{lem:key} that we need is where $k=1$ and $X$ is a ball. (Only for the proof of the Weyl law in Theorem~\ref{thm:weyl} below will we need a more general case where $k>1$ and $X$ is a disjoint union of balls.) This case of the lemma asserts that under the hypotheses of the lemma, if the ball $B(r)$ can be symplectically embedded into $M$, then
\begin{equation}
\label{eqn:k1case}
c_{U\sigma}(\phi,\gamma,H_1) \le c_\sigma(\phi,\gamma,H_2) + \int_\gamma(H_2-H_1)dt - r.
\end{equation}
\end{remark}

\begin{proof}[Proof of Lemma~\ref{lem:key}.]
By the continuity in Corollary~\ref{cor:continuity}, we can assume without loss of generality, by slightly decreasing $H_1$ and increasing $H_2$ if necessary, that $H_1<H_2$ and that $\phi_{H_1}$ and $\phi_{H_2}$ are nondegenerate. We now proceed in three steps.

{\em Step 1.}
Choose $J_1$, $J_2$, and $J$ as in Step 3 of the proof of Proposition~\ref{prop:invariance}. Also, fix $k$ distinct points $z_1,\ldots,z_k\in \overline{M}$.

We claim that the composition
\begin{equation}
\label{eqn:Ucomp}
U^k\circ \Psi_{H_1,H_2}=\Psi_{H_1,H_2}\circ U^k: HP(\phi_{H_2},\gamma_{H_2},G) \longrightarrow HP(\phi_{H_1},\gamma_{H_1},G)
\end{equation}
is induced by a (noncanonical) chain map
\[
\psi: (CP(\phi_{H_2},\gamma_{H_2},G),\partial_{J_2}) \longrightarrow (CP(\phi_{H_1},\gamma_{H_1},G),\partial_{J_1})
\]
with the following property: Similarly to \eqref{eqn:chainmap}, we can write $\psi$ in the form
\begin{equation}
\label{eqn:newchainmap}
\psi\sum_{\alpha,Z}n_{\alpha,Z}(\alpha,Z) = \sum_{\beta,W}\left(\sum_\alpha \sum_{\substack{V\in H_2(Y_\phi,\alpha,\beta)\\I(\alpha,\beta,V)=2k}}n_{\alpha,W+V}m_{\alpha,\beta,V}\right)(\beta,W),
\end{equation}
such that:
\begin{description}
\item{(*)}
If the coefficient $m_{\alpha,\beta,V}\neq 0$, then there is a broken $J$-holomorphic current in $\overline{M}$ from $\alpha$ to $\beta$ in the relative homology class $V$ passing through the points $z_1,\ldots,z_k$.
\end{description}
The reason is that in the proof of the ``holomorphic curves axiom'' in \cite{chen}, the chain map counts solutions to the Seiberg-Witten equations on $\overline{M}$ perturbed using a large multiple of the symplectic form; and as the perturbation goes to infinity, the zero set $\alpha$ for a sequence of Seiberg-Witten solutions converge to a holomorphic current. Here $\alpha$ denotes the component of the spinor in the positive imaginary eigenspace of Clifford multiplication by the symplectic form on $\overline{M}$. Similarly, as in \cite{taubes5} (see the review in \S\ref{sec:Umap}), the map \eqref{eqn:Ucomp} can be induced by a chain map\footnote{
If one chooses different points $z_1',\ldots,z_k'$ in $\overline{M}$, one obtains a chain homotopic chain map. One can define a chain homotopy by choosing paths $\eta_i$ in $\overline{M}$ from $z_i$ to $z_i'$ and counting Seiberg-Witten solutions where $\alpha$ vanishes somewhere on $\eta_i$ for each $i=1,\ldots,k$. If one moves the points far up on the positive end of $\overline{M}$, one obtains a chain map corresponding to $\Psi_{H_1,H_2}\circ U^k$; and if one moves the points far down on the negative end of $\overline{M}$, one obtains a chain map corresponding to $U^k\circ\Psi_{H_1,H_2}$.
}
counting Seiberg-Witten solutions where $\alpha$ is constrained to vanish at the points $z_1,\ldots,z_k$.  As the perturbation goes to infinity, the zero sets of $\alpha$ for the Seiberg-Witten solutions converge to a holomorphic current passing through the points $z_1,\ldots,z_k$.

{\em Step 2.\/} Define $\Delta$ as in equation \eqref{eqn:Delta}. We claim that by making suitable choices of $J$ and $z_1,\ldots,z_k$, we can arrange for the chain map \eqref{eqn:newchainmap} to have the following property:
\begin{description}
\item{(**)}
If the coefficient $m_{\alpha,\beta,Z}\neq 0$, then the inequality \eqref{eqn:step4} can be refined to
\[
\int_W\omega_{\phi_{H_1}} \le \int_{W+V}\omega_{\phi_{H_2}} + \Delta - c_k^{\op{Alt}}(X,\omega).
\]
\end{description}

To prove this, by the $C^0$ continuity of $c_k^{\op{Alt}}$, we can assume without loss of generality that $(X,\omega)$ is nondegenerate. Write $Y=\partial X$, and let $\lambda$ denote the contact form on $Y$. We will identify $X$ with its image under the symplectic embedding into $M$. We can remove $X$ from $\overline{M}$ and attach symplectization ends to form a new completed cobordism
\begin{equation}
\label{eqn:Mbarprime}
\overline{M}' = \left((-\infty,0]\times Y\right) \cup_Y \left(\overline{M}\setminus X\right).
\end{equation}
Choose an almost complex structure $J_X\in\mathcal{J}(\overline{X})$ as in \eqref{eqn:ckalt}. We can then choose the almost complex structure $J$ on $M$ so that it glues to $J_X$ in \eqref{eqn:Mbarprime}, to give a well-defined almost complex structure $J'$ on $\overline{M}'$. We can further choose a sequence of $\omega_M$-compatible almost structures $J(n)$ on $M$ such that $J(n)$ agrees with $J$ outside of $X$, and moreover inside of $X$, the boundary has a neighborhood that can be identified with $(-n,0]\times\partial Y$ so that $J(n)$ agrees with $J'$.

By placing the points $z_1,\ldots,z_k$ inside $X$, and using a compactness argument as in \cite[Lem.\ 3]{altspec}, we can arrange the following: For any sequence of broken $J(n)$-holomorphic currents as in (*), with an upper bound on $\int_{W+V}\omega_{\phi_{H_1}}-\int_W\omega_{\phi_{H_2}}$, after passing to a subsequence the following hold:
\begin{itemize}
\item
The holomorphic currents in $\overline{M}$ converge on compact sets to a $J'$-holomorphic current in $\overline{M}'$, which on $(-\infty,0]\times\partial B_i$ is asymptotic to a finite multiset $\alpha_M$ of Reeb orbits.
\item
The homolomorphic currents in $X$ converge on compact sets to a $J_X$-holomorphic curve $u\in\mathcal{M}^{J_X}(\overline{X};z_1,\ldots,z_k)$, asymptotic to a finite multiset $\alpha_X$ of Reeb orbits.
\item
If $\alpha_X\neq\alpha_M$, then there is a broken $J_X$-holomorphic current in $\R\times Y$ from $\alpha_M$ to $\alpha_X$, as in \cite[\S5.3]{bn}. In particular, in all cases we have
\begin{equation}
\label{eqn:alphaX}
\mathcal{E}(u) = \int_{\alpha_X} \lambda \le \int_{\alpha_M}\lambda.
\end{equation}
\end{itemize}
Repeating Step 4 of the proof of Proposition~\ref{prop:invariance} then shows that
\[
\int_{\alpha_M}\lambda+ \int_W\omega_{\phi_{H_1}} \le \int_{W+V}\omega_{\phi_{H_2}} + \Delta.
\]
By the inequality \eqref{eqn:alphaX}, it follows that
\[
\mathcal{E}(u) + \int_W\omega_{\phi_{H_1}} \le \int_{W+V}\omega_{\phi_{H_2}} + \Delta.
\]
Since $J_X$ and $z_1,\ldots,z_k$ were arbitrary, it follows from the definition \eqref{eqn:ckalt} that we can choose $J_X$ and $z_1,\ldots,z_k$ so as to replace $\mathcal{E}(u)$ by $c_k^{\op{Alt}}(X,\omega)$ in the above inequality. It then follows that (**) holds with $J=J(n)$ if $n$ is sufficiently large.

{\em Step 3.} If we make the choices as in Step 2, then for each $L\in\R$, the chain map \eqref{eqn:newchainmap} restricts to a chain map
\[
\psi: \left(CP^L(\phi_{H_2},\gamma_{H_2},G),\partial_{J_2}\right) \longrightarrow \left(CP^{L+\Delta-c_k^{\op{Alt}}(X,\omega)}(\phi_{H_1},\gamma_{H_1},G),\partial_{J_1}\right)
\]
The induced map on homology fits into a commutative diagram
\[
\begin{CD}
HP^L(\phi_{H_2},\gamma_{H_2},G) @>{\imath^L}>> HP(\phi_{H_2},\gamma_{H_2},G)\\
@VVV @VV{U^k\circ\Psi_{H_1,H_2}}V\\
HP^{L+\Delta-c_k^{\op{Alt}}(X,\omega)}(\phi_{H_1},\gamma_{H_1},G) @>{\imath^{L+\Delta-c_k^{\op{Alt}}(X,\omega)}}>> HP(\phi_{H_1},\gamma_{H_1},G).
\end{CD}
\]
We are now done as in the proof of Proposition~\ref{prop:spectralchange}.
\end{proof}



\section{From spectral gaps to periodic orbits}
\label{sec:gap}

We now explain a mechanism for detecting the creation of periodic orbits. The following concept will be useful:

\begin{definition}
\label{def:gap}
Let $\phi$ be an area-preserving diffeomorphism of $(\Sigma,\omega)$ and let $d$ be an integer with $d>g$. Define the {\em minimal spectral gap\/}
\[
\op{gap}_d(\phi)\in[0,\infty]
\]
to be the infimum, over reference cycles $\gamma$ for $\phi$ with $d(\gamma)=d$, subgroups $G\subset\Ker([\omega_\phi])$, and classes $\sigma\in HP(\phi,\gamma,G)$ with $U\sigma\neq 0$, of $c_\sigma(\phi,\gamma)-c_{U\sigma}(\phi,\gamma)$.
\end{definition}

In the above definition note that $c_\sigma(\phi,\gamma)-c_{U\sigma}(\phi,\gamma)\ge 0$ by \eqref{eqn:positivearea}. We now have the following relation between spectral gaps and creation of periodic orbits.

\begin{proposition}
\label{prop:gap}
Let $\phi$ be an area-preserving diffeomorphism of $(\Sigma,\omega)$, and suppose that the Hamiltonian isotopy class $[\phi]$ is rational. Let $\mc{U}\subset\Sigma$ be a nonempty open set and let $H$ be a $(\mc{U},a,l)$-admissible Hamiltonian as in Definition~\ref{def:ual}. Let $d$ be an integer with $d>g$, and suppose that
\begin{equation}
\label{eqn:gaphypo}
\op{gap}_d(\phi)< a.
\end{equation}
Then for some $\tau\in[0,l^{-1}\op{gap}_d(\phi)]$, the map $\phi_{\tau H}$ has a periodic orbit intersecting $\mc{U}$ with period $\le d$.
\end{proposition}

\begin{proof}
We proceed in four steps.

{\em Step 1.} We first claim that for any class $\sigma\in HP(\phi,\gamma,G)$ with $U\sigma\neq 0$, and for any $\delta\ge 0$, we have
\begin{equation}
\label{eqn:ballgap}
c_{U\sigma}(\phi,\gamma) \le c_\sigma(\phi,\gamma,\delta H) + \delta\int_\gamma H\,dt - \op{min}(\delta l,a).
\end{equation}
Here we are using the convention of Notation~\ref{spectralnotation} on the right hand side.

To prove \eqref{eqn:ballgap}, recall from Definition~\ref{def:ual} that there is a disk $D\subset\mc{U}$ of area $a$ and an interval $I\subset(0,1)$ of length $l$ such that $H\ge 1$ on $I\times D$. We can regard $H$ as defined on $Y_\phi$ as in \S\ref{sec:invariance}. Let $M_\delta$ denote the cobordism \eqref{eqn:M} between the mapping torus $Y_{\phi}$ and the graph of $\delta H$. Then we can symplectically embed the polydisk $P(a,\delta l)$, namely the symplectic product of two-disks of areas $a$ and $\delta l$, into $M_\delta$. Consequently, we can symplectically embed the ball $B(\op{min}(\delta l,a))$ into $M_\delta$. The inequality \eqref{eqn:ballgap} now follows from the $k=1$ case of Lemma~\ref{lem:key}; see Remark~\ref{rem:oneball}.

{\em Step 2.} Suppose now that $\delta\ge 0$ and
\begin{description}
\item{(*)} for all $\tau\in[0,\delta]$, the map $\phi_{\tau H}$ has no periodic orbit intersecting $\mc{U}$ with period $\le d$.
\end{description}
We claim that if $\gamma$ is a reference cycle for $\phi$ with $d(\gamma)=d$, and if $\sigma\in HP(\phi,\gamma,G)$ is a nonzero class, then
\begin{equation}
\label{eqn:spectrumconstant}
c_\sigma(\phi,\gamma,\delta H) = c_\sigma(\phi,\gamma) - \delta\int_\gamma H\,dt.
\end{equation}

To prove \eqref{eqn:spectrumconstant}, let $S$ denote the set of actions of $(G,\gamma)$-anchored orbit sets for $\phi$. By the hypothesis (*), if $\tau\in[0,\delta]$, then the set of actions of $(G,\gamma)$-anchored orbit sets for $\phi_{\tau H}$ is $S-\tau\int_\gamma H\,dt$.  Since we are assuming that $[\phi]$ is rational, it follows from Proposition~\ref{prop:representation} that the function
\begin{equation}
\label{eqn:function}
\begin{split}
[0,\delta] &\longrightarrow \R,\\
\tau &\longmapsto c_\sigma(\phi,\gamma,\tau H) + \tau\int_\gamma H\,dt
\end{split}
\end{equation}
takes values in the set $S$. This function is also continuous by Corollary~\ref{cor:continuity}. However the set $S$ has measure zero as in \cite[Lem.\ 2.2]{irie1}. It follows that the function \eqref{eqn:function} is constant, and this proves \eqref{eqn:spectrumconstant}.

{\em Step 3.} We now show that if $\delta > l^{-1}\op{gap}_d(\phi)$, then for some $\tau\in[0,\delta]$, the map $\phi_{\tau H}$ has a periodic orbit intersecting $\mc{U}$ with period $\le d$. Suppose to get a contradiction that (*) holds. Suppose that $d(\gamma)=d$ and that $\sigma\in HP(\phi,\gamma,G)$ satisfies $U\sigma\neq 0$. Then combining \eqref{eqn:ballgap} and \eqref{eqn:spectrumconstant} gives
\[
c_\sigma(\phi,\gamma) - c_{U\sigma}(\phi,\gamma) \ge \op{min}(\delta l,a).
\]
It then follows from Definition~\ref{def:gap} that
\[
\op{gap}_d(\phi) \ge \op{min}(\delta l, a).
\]
Since we assumed that $\delta l > \op{gap}_d(\phi)$, this contradicts the hypothesis \eqref{eqn:gaphypo}.

{\em Step 4.} The proposition follows from Step 3 by replacing $\mc{U}$ by an open set $\mc{V}$ such that $\overline{V}\subset\mc{U}$ and $H$ is supported in $[0,1]\times \mc{V}$, and using a compactness argument.
\end{proof}

As an example of the significance of Proposition~\ref{prop:gap}, we have the following corollary, which is a PFH analogue of \cite[Lem.\ 3.1]{cgm} for ECH:

\begin{corollary}
Let $\phi$ be an area-preserving diffeomorphism of $(\Sigma,\omega)$, and suppose that the Hamiltonian isotopy class $[\phi]$ is rational. Let $d$ be an integer with $d>g$ and suppose that $\op{gap}_d(\phi)=0$. Then every point in $\Sigma$ is contained in a periodic orbit of $\phi$ with period $\le d$. In particular, $\phi$ is periodic with period $\le d!$.
\end{corollary}

\begin{proof}
It follows immediately from Proposition~\ref{prop:gap} that every nonempty open set $\mc{U}\subset\Sigma$ contains a periodic point of period $\le d$. It then follows from a compactness argument that every point in $\Sigma$ is periodic with period $\le d$.
\end{proof}


\section{Proofs of theorems}
\label{sec:proofs}

We now prove all of our theorems stated in \S\ref{sec:intro}. We begin with the following simple observation:

\begin{lemma}
\label{lem:gapbound}
Suppose that $HP(\phi,\gamma,G)$ contains $U$-cyclic elements. Write $d=d(\gamma)$ and $A=\int_\Sigma\omega$. Then
\[
\op{gap}_d(\phi) \le \frac{A}{d-g+1}.
\]
\end{lemma}

\begin{proof}
We are given that equation \eqref{eqn:Ucyclic} holds for some positive integer $m$. It follows using Proposition~\ref{prop:spectralproperties}(a) that
\[
\begin{split}
mA &= c_{\sigma}(\phi,\gamma) - c_{U^{m(d-g+1)}\sigma}(\phi,\gamma)\\
&= \sum_{i=1}^{m(d-g+1)}\big(c_{U^{i-1}\sigma}(\phi,\gamma) - c_{U^i\sigma}(\phi,\gamma)\big).
\end{split}
\]
Since each of the summands on the right hand side is nonnegative, at least one of them must be less than or equal to $A/(d-g+1)$.
\end{proof}

\begin{proof}[Proof of Theorem~\ref{thm:closing}.]
Suppose that $[\phi]$ is rational and satisfies the $U$-cycle property. Let $\mc{U}\subset\Sigma$ be a nonempty open set. We need to show that there is a $C^\infty$ small Hamiltonian perturbation, supported in $\mc{U}$, of $\phi$ to a map having a periodic orbit intersecting $\mc{U}$.

Let $H$ be a $(\mc{U},a,l)$-admissible Hamiltonian. It is enough to show that for all $\delta>0$, there exists $\tau\in[0,\delta]$ such that $\phi_{\tau H}$ has a periodic orbit intersecting $\mc{U}$.

Since $[\phi]$ has the $U$-cycle property, it follows from Lemma~\ref{lem:gapbound} that
\[
\liminf_{d\to\infty}\op{gap}_d(\phi) = 0.
\]
Thus we can find $d>g$ such that $\op{gap}_d(\phi)<\min(a,l\delta)$. For such $d$, since $[\phi]$ is rational, Proposition~\ref{prop:gap} implies that for some $\tau\in[0,\delta]$, the map $\phi_{\tau H}$ has a periodic orbit intersecting $\mc{U}$ of period at most $d$. 
\end{proof}

As noted in \S\ref{sec:cl}, Corollary~\ref{cor:torus} follows from Theorem~\ref{thm:closing} and the following lemma:

\begin{lemma}
\label{lem:Ucycle}
Let $\phi$ be an area-preserving diffeomorphism of $T^2$, and suppose that the Hamiltonian isotopy class $[\phi]$ is rational, i.e.\ $[\omega_\phi]$ is a positive multiple of the image of an integral class $\Omega\in H^2(Y_\phi;\Z)$. Then $[\phi]$ has the $U$-cycle property. In fact, if $\Gamma$ is a positive integer multiple of $\op{PD}(\Omega)$, then $HP(\phi,\Gamma,\Ker([\omega_\phi]))\neq 0$, and every nonzero element of the latter is $U$-cyclic of order $\le 6$.
\end{lemma}

\begin{proof}
Let $\phi$ be an area-preserving diffeomorphism of $(\Sigma,\omega)$, of arbitrary genus for now, and assume that $[\phi]$ is rational. Since the cohomology class $[\omega_\phi]\in H^2(Y_\phi;\R)$ is a real multiple of the image of an integral cohomology class $\Omega\in H^2(Y_\phi;\Z)$, we can find classes $\Gamma\in H_1(Y_\phi)$ with $d(\Gamma)$ arbitrarily large such that the pair $(\phi,\Gamma)$ is monotone as in Definition~\ref{def:monotone}. (Simply take $\Gamma$ to the Poincar\'e dual of $n\Omega - c_1(E)/2$ where $n$ is a large integer.)

For such a $\Gamma$, a result of Lee-Taubes \cite[Cor.\ 1.3]{lt}, tensored with $\Z/2$, asserts that if $g>0$ and $d(\Gamma)>2g-2$, then we have the following variant of the isomorphism \eqref{eqn:leetaubes}:
\begin{equation}
\label{eqn:leetaubes2}
\overline{HP}(\phi,\Gamma) \simeq \overline{HM}^{-*}(Y_\phi,\frak{s}_\Gamma,c_b;\Z/2).
\end{equation}
Here the left hand side is the untwisted PFH from \S\ref{sec:monotone}. The right hand side is an instance of the ``bar'' version of Seiberg-Witten Floer cohomology, with the ``balanced'' perturbation, defined by Kronheimer-Mrowka \cite[\S30]{km}. As with \eqref{eqn:leetaubes}, the isomorphism \eqref{eqn:leetaubes2} intertwines the $U$ maps on both sides, as discussed in \S\ref{sec:Umap}.

If $\Sigma=T^2$, then for $\Gamma$ as above, by computations in \cite[\S35.3]{km} (see the remark after \cite[Cor.\ 1.3]{lt}), we always have $\overline{HP}(\phi,\Gamma)\neq 0$, with the graded pieces $\overline{HP}_i(\phi,\gamma)$ having rank $\le 2$, on which $U$ acts as an isomorphism. In particular, $U^{d(\Gamma)}$ is a permutation of the set of nonzero elements in each graded piece $\overline{HP}_i(\Phi,\Gamma)$. Then there is a positive integer $m\le 6$ such that $U^{md(\Gamma)}$ is the identity on all such groups $\overline{HP}(\phi,\Gamma)$. We are now done by Lemma~\ref{lem:twisted}.
\end{proof}

\begin{remark}
\label{rem:upgrade}
The paper \cite{ucyclic}, which appeared after the original version of this paper, studies $U$-cyclic elements in more detail and generality. In particular, the proof of \cite[Thm.\ 1]{ucyclic} shows that if $R$ is any coefficient ring then
\[
(1-U^{d(\Gamma)-g+1})^{b_1(Y_\phi)+1}\overline{HM}^{-*}(Y_\phi,\frak{s}_\Gamma,c_b;R)=0.
\]
In our case where $R=\Z/2$, it follows that $U^{m(d(\Gamma)-g+1)}$ equals the identity, where $m=b_1(Y_\phi)+1$ when $b_1(Y_\phi)$ is odd, and $m=b_1(Y_\phi)+2$ when $b_1(Y_\phi)$ is even. In addition, it was shown in \cite{rohil}, which appeared simultaneously with the original version of this paper, that
\[
\overline{HM}^{-*}(Y_\phi,\frak{s}_\Gamma,c_b;R)\neq 0.
\]
As a result, Lemma~\ref{lem:Ucycle} can be upgraded to assert the following:  Let $\phi$ be an area-preserving diffeomorphism of $\Sigma$, and suppose that $[\omega_\phi]$ is a positive multiple of the image of $\Omega\in H^2(Y_\phi;\Z)$. If $\Gamma$ is a positive integer multiple of $\op{PD}(\Omega)$, then $HP(\phi,\Gamma,\Ker([\omega_\phi]))\neq 0$, and every nonzero element of the latter is $U$-cyclic of order $\le m$.
\end{remark}

To prove Theorems~\ref{thm:quant0} and \ref{thm:quant1}, we first prove a more general statement:

\begin{theorem}
\label{thm:quantitative}
Let $\phi$ be an area-preserving diffeomorphism of $(\Sigma,\omega)$ such that the Hamiltonian isotopy class $[\phi]$ is rational. Suppose that there exists a positive integer $d_0$ such that $\phi$ has $U$-cyclic elements of degree $d$ whenever $d$ is a positive multiple of $d_0$ with $d>g$. Let $\mc{U}\subset\Sigma$ be a nonempty open set and write $A=\int_\Sigma\omega$. Let $H$ be a $(\mc{U},a,l)$-admissible Hamiltonian. If $\delta l\le a$, then for some $\tau\in [0,\delta]$, the map $\phi_{\tau H}$ has a periodic orbit intersecting $\mc{U}$ with period at most $d_0k$, where
\begin{equation}
\label{eqn:quantitative}
k=\floor{\frac{A\delta^{-1}l^{-1}+g-1}{d_0}}+1.
\end{equation}
\end{theorem}

\begin{proof}
Write $d=kd_0$. It follows from \eqref{eqn:quantitative} that
\begin{equation}
\label{eqn:smallestinteger}
d > A\delta^{-1}l^{-1}+g-1.
\end{equation}
In particular, it follows from \eqref{eqn:smallestinteger} that $d>g$, since $\delta l \le a < A$.
Then by Lemma~\ref{lem:gapbound}, we have
\[
\op{gap}_{d}(\phi) \le \frac{A}{d-g+1}.
\]
Also, it follows from the above inequalities that $\op{gap}_{d}(\phi) < \delta l \le a$. It then follows from Proposition~\ref{prop:gap}  that for some $\tau\in[0,l^{-1}\op{gap}_{d}(\phi)]\subset [0,\delta]$, the map $\phi_{\tau H}$ has a periodic orbit intersecting $\mc{U}$ with period at most $d$.
\end{proof}

\begin{proof}[Proof of Theorem~\ref{thm:quant0}.]
If $\Sigma=S^2$, then by Example~\ref{ex:US3}, $\phi$ has $U$-cyclic elements of degree $d$ for all positive integers $d$. Thus Theorem~\ref{thm:quantitative} applies with $d_0=1$ to give the result.
\end{proof}

\begin{proof}[Proof of Theorem~\ref{thm:quant1}.]
This follows from Theorem~\ref{thm:quantitative} and Lemma~\ref{lem:Ucycle}.
\end{proof}


\section{Asymptotics of PFH spectral invariants}
\label{sec:weyl}

To conclude, we now prove the following ``Weyl law'' for PFH spectral invariants. Under certain hypotheses, it describes how the difference in spectral invariants $c_\sigma(\phi,\gamma,H_2)-c_\sigma(\phi,\gamma,H_1)$ behaves as the degree $d(\gamma)\to\infty$.

\begin{theorem}
\label{thm:weyl}
Let $\phi$ be a (possibly degenerate) area-preserving diffeomorphism of $(\Sigma,\omega)$. Let $\{G_i\}_{i\ge 1}$ be a sequence of subgroups of $\Ker([\omega_\phi])$, let $\{\gamma_i\}_{i\ge 1}$ be a sequence of reference cycles for $\phi$, and for each $i\ge 1$, let $\sigma_i\in HP(\phi,\gamma_i,G_i)$ be a nonzero class. Assume that:
\begin{itemize}
\item $\lim_{i\to\infty}d(\gamma_i)=\infty$.
\item There is a positive integer $m$ such that each class $\sigma_i$ is $U$-cyclic of order $\le m$.
\end{itemize}
Let $H_1,H_2:Y_\phi\to\R$. Write $A=\int_\Sigma\omega$. Then
\[
\boxed{
\lim_{i\to\infty}\frac{c_{\sigma_i}(\phi,\gamma_i,H_2) - c_{\sigma_i}(\phi,\gamma_i,H_1) + \int_{\gamma_i} (H_2-H_1)dt}{d(\gamma_i)} = A^{-1} \int_{Y_\phi}(H_2-H_1)\omega_\phi\wedge dt.
}
\]
\end{theorem}

\begin{remark}
If $\phi$ is rational, then Theorem~\ref{thm:weyl} is not vacuous. In this case, as explained in Remark~\ref{rem:upgrade}, one can find a sequence of nonzero classes $\sigma_i\in HP(\phi,\gamma_i,\op{Ker}([\omega_\phi]))$ with $d(\gamma_i)\to\infty$, and each $\sigma_i$ is automatically $U$-cyclic of order $\le b_1(Y_\phi)+2$.
\end{remark}

\begin{example}
\label{ex:US4}
Let $D$ be a disk with a symplectic form $\omega$ of area $1$, and let $\phi$ be the time $1$ map of a Hamiltonian $H:[0,1]\times D\to\R$ which vanishes on $[0,1]\times\{x\}$ when $x$ is near $\partial D$.  Then $\phi$ defines an area-preserving diffeomorphism of $S^2$, with a symplectic form of area $1$, which is the identity on an open set. Recall from Example~\ref{ex:US2} that if $\gamma=d[S^1]\times\{x\}$, where $x$ corresponds to a point on $\partial D$, then $HP(\phi,\gamma,\{0\})$ is the free $\Lambda$-module generated by classes $e_{d,0},\ldots,e_{d,d}$; and each of these classes is $U$-cyclic of order $1$. Note that if $\phi$ is the identity, then each spectral invariant $c_{e_{d,i}}(\phi,\gamma)=0$. It then follows from Theorem~\ref{thm:weyl} that in general, we have
\begin{equation}
\label{eqn:weyldisk}
\lim_{d\to\infty}\frac{c_{e_{d,0}}(\phi,d[S^1]\times\{x\})}{d} = \int_{[0,1]\times D}H\omega\wedge dt.
\end{equation}
Here the right hand side (up to constant factors depending on conventions) is the {\em Calabi invariant\/} of $\phi$; see e.g.\ \cite{gg}.

The special case of \eqref{eqn:weyldisk} where $\phi$ is a ``monotone twist'' was proved\footnote{The paper \cite{simple} writes a slightly different, but equivalent, version of \eqref{eqn:weyldisk}. That paper defines spectral invariants using the variant $\widetilde{HP}(\phi,\gamma,\{0\})$ from \S\ref{sec:monotone}, which is possible here since monotonicity holds. Our spectral invariant $c_{e_{d,i}}(\phi,d[S^1]\times \{x\})$ agrees with the spectral invariant denoted in \cite{simple} by $c_{d,2i-d}(\phi)$.} by direct calculation in \cite[Thm.\ 5.1]{simple}, and this result plays a key role in the proof of the simplicity conjecture. It is also noted in \cite[\S7.4]{simple} that \eqref{eqn:weyldisk} implies that the Calabi invariant extends to a homomorphism defined on the group of compactly supported ``hameomorphisms'' of the disk. The latter statement was subsequently proved using different methods in \cite[Thm.\ 1.4]{fiveauthors}.
\end{example}

\begin{proof}[Proof of Theorem~\ref{thm:weyl}.]
We use a ``ball packing'' argument similar to \cite[\S3.2]{vc}.

Suppose first that $H_1<H_2$. Let $X$ be a finite disjoint union of balls symplectically embedded in $M$ and write $V=\op{vol}(X)$.

We know that each $\sigma_i$ is $U$-cyclic of order $m_i$ where $m_i\le m$. By Proposition~\ref{prop:spectralproperties}(a), if $k_i=m_in_i(d(\gamma_i)-g+1)$ where $n_i$ is an integer, then we have
\[
c_{U^{k_i}\sigma_i}(\phi,\gamma_i,H_1) = c_{\sigma_i}(\phi,\gamma_i,H_1) -Am_in_i.
\]
Then Lemma~\ref{lem:key} gives
\begin{equation}
\label{eqn:fromkey}
c_{\sigma_i}(\phi,\gamma_i,H_2) - c_{\sigma_i}(\phi,\gamma_i,H_1) + \int_{\gamma_i}(H_2-H_1)dt \ge  c_{k_i}^{\op{Alt}}(X) - Am_in_i.
\end{equation}
Now choose
\[
n_i=\floor{\frac{d(\gamma_i)^2V}{m_iA^2(d(\gamma_i)-g+1)}}.
\]
Then it follows from \eqref{eqn:fromkey} that
\[
\begin{split}
\liminf_{i\to\infty}
\frac{c_{\sigma_i}(\phi,\gamma_i,H_2) - c_{\sigma_i}(\phi,\gamma_i,H_1) + \int_{\gamma_i}(H_2-H_1)dt}{d(\gamma_i)} &\ge \liminf_{i\to\infty}\frac{c_{k_i}^{\op{Alt}}(X) - Am_in_i}{d(\gamma_i)} \\
&= A^{-1}V
\end{split}
\]
Here in the second line we have used \eqref{eqn:echvol} and the hypothesis that $d(\gamma_i)\to\infty$.

Now we can choose $X$ to make $V$ arbitrarily close to
\[
\begin{split}
\op{vol}(M,\omega_M) &= \frac{1}{2}\int_M\omega_M\wedge\omega_M\\
&= \int_{Y_\phi}(H_2-H_1)\omega_\phi\wedge dt.
\end{split}
\]
Thus we obtain
\[
\liminf_{i\to\infty}\frac{c_{\sigma_i}(\phi,\gamma_i,H_2) - c_{\sigma_i}(\phi,\gamma_i,H_1) + \int_{\gamma_i} (H_2-H_1)dt}{d(\gamma_i)} \ge A^{-1} \int_{Y_\phi}(H_2-H_1)\omega_\phi\wedge dt.
\]

By Remark~\ref{rem:constantshift}, both sides of the above inequality change by the same amount if one adds a constant to $H_1$ or $H_2$. Thus the above inequality is true for any $H_1$ and $H_2$, without the hypothesis that $H_1<H_2$. In particular, the above inequality is true with $H_1$ and $H_2$ switched, which gives
\[
\limsup_{i\to\infty}\frac{c_{\sigma_i}(\phi,\gamma_i,H_2) - c_{\sigma_i}(\phi,\gamma_i,H_1) + \int_{\gamma_i} (H_2-H_1)dt}{d(\gamma_i)} \le A^{-1} \int_{Y_\phi}(H_2-H_1)\omega_\phi\wedge dt.
\]
The above two inequalities imply the theorem.
\end{proof}

\begin{remark}
By choosing the ball packings carefully, as in the proof of \cite[Thm.\ 1.1]{ruelle}, one can show that the rate of convergence in Theorem~\ref{thm:weyl} is $O(d(\gamma_i)^{-1/2})$.
\end{remark}


\end{document}